\documentclass[12pt, twoside]{article}
\usepackage{a4wide}
\usepackage{amssymb}
\usepackage{amsmath}
\usepackage{amsfonts}
\usepackage{graphicx}
\usepackage{color}
\usepackage{subfigure}
\usepackage{bbm}

\newtheorem{thm}{Theorem}[section]
\newtheorem{cor}[thm]{Corollary}
\newtheorem{lem}[thm]{Lemma}
\newtheorem{defin}[thm]{Definition}
\newtheorem{rem}[thm]{Remark}

\providecommand{\U}[1]{\protect\rule{.1in}{.1in}}

\newcommand{\subjclass}[1]{\bigskip\noindent\emph{2010 Mathematics Subject Classification:}\enspace#1}
\newcommand{\keywords}[1]{\noindent\emph{Keywords:}\enspace#1}

\newenvironment{proof}[1][Proof]{\noindent\textbf{#1.} }{\ \rule{0.5em}{0.5em}}

\catcode`\@=11

\begin{document}

\title{
Interfaces determined by capillarity and gravity in a two-dimensional porous medium
}

\author{
M. Calle\thanks{
Universidad Carlos III de Madrid (UC3M),
 Departamento de Econom\'ia,
 Facultad de Ciencias Sociales y Jur\'idicas, 
Calle Madrid, 126. 28903 Getafe (Madrid) Spain. 
E-mail: maria.calle.garcia@uc3m.es}, 
 C.~M. Cuesta\thanks{
University of the Basque Country (UPV/EHU),
Departmento de Matem\'aticas, 
Aptdo. 644, 48080 Bilbao, Spain. 
E-mail: carlotamaria.cuesta@ehu.es}, 
J.~J.~L. Vel\'{a}zquez\thanks{
Institut f\"{u}r Angewandte Mathematik,Universit\"{a}t Bonn,
Endenicher Allee 60, 53115 Bonn, Germany. 
E-mail: velazquez@iam.uni-bonn.de}
}


%
%
\date{}
\maketitle
\begin{abstract}
We consider a two-dimensional model of a porous medium where circular grains are uniformly 
distributed in a squared container. We assume that such medium is partially filled with 
water and that the stationary interface separating the water phase from the air phase is 
described by the balance of capillarity and gravity. 
Taking the unity as the average distance between grains, 
we identify four asymptotic regimes that depend on the Bond number and 
the size of the container. We analyse, in probabilistic terms, 
the possible global interfaces that can form in each of these regimes. 
In summary, we show that in the regimes where gravity dominates 
the probability of configurations of grains allowing solutions close to 
the horizontal solution is close to one. Moreover, in such regimes where 
the size of the container is sufficiently large we can describe deviations 
from the horizontal in probabilistic terms. 
 On the other hand, when capillarity dominates while 
the size of the container is sufficiently large, we find that 
the probability of finding interfaces close to the graph of a given smooth 
curve without self-intersections is close to one.   


\end{abstract}

\section{Introduction}\label{intro}
In this paper, we present a two-dimensional model of a porous medium that consists 
of a squared container partially filled with a liquid and where the pore throats are 
formed by fixed circular grains that are uniformly distributed. 
We analyse in probabilistic terms the structure of the stationary interfaces that 
separate the liquid from air. We  consider several parameter regimes where, essentially, 
either gravity or capillarity dominates. 

In this model the form of the interfaces, both at the pore scale and at the macro scale,
 is determined by the balance between gravity and capillary forces.  
For a given volume of liquid, these interfaces are the union of single 
interfaces (see \cite{Finn}) that meet the solid matrix at given contact 
angles following Young's law (see e.g. \cite{batchelor}).

The quasi-stationary system seems to evolve by means of a 
sequence of local minimizers of the energy. The energy for steady states is the sum of 
surface energy and potential gravitational energy. The contact angle conditions can be 
encoded in the different interfacial energies. One of the most relevant features in these 
flows is the existence of a large number of equilibria that are local minimizers of 
the total energy of the system. In quasi-stationary situations, 
if the volume of the liquid phase changes slowly, transitions occur between different minimizers. 
Understanding the evolution of such a system would require the 
description of the global motion of the interface that explains the jumps 
between these multiple local energy minimizers. A first step towards this goal 
consists in understanding the structure of the minimizers.

The aim of this paper is thus to describe the structure of equilibrium 
interfaces for different parameter regimes in the model described above. 
There are several characteristic lengths in the problem, namely, the distance 
between grains, their size, as well as the macroscopic size of the system. 
On the other hand, there are other length scales (the capillary length and 
the inverse of the Bond number) that, as we shall see, give the distances where surface tension 
and gravity balance (typically, surface tension dominates for 
small distances and gravity for large ones).

This study is motivated by the experiments presented in Furuberg-Maloy-Feder 
\cite{Maloy92} and \cite{Maloy96}. The authors reproduced Haines
 jumps (see \cite{haines}) in a slow drainage process in a container consisting 
of two plates separated by thin cylinders that act as obstacles or `grains'. 
We recall that Haines jumps are abrupt changes of the pressure in the liquid 
phase due to sudden changes in the geometry of the interface. 
From the mathematical point of view these jumps seem to correspond to a 
fast transition from one equilibrium state to a different one in the 
capillarity equations. 

We mention that there are several mathematical studies 
related to the study of Multiphase Flow in Porous Media. 
However, much of the mathematical analysis has focused on the qualitative 
study of semi-empirical macroscopic descriptions of these flows 
(see e.g. \cite{bear}). There are also rigorous results 
concerning the derivation of such macroscopic laws from the flows taking place
 at the pore scale. Not surprisingly, such attends are restricted to 
single-phase flow (see the chapter of \cite{SP}, \cite{Tartar}) or simplified models mimicking 
essential features of multi-phase flow (e.g. \cite{JBS07}, \cite{BS05_I}, \cite{BS05_II}, \cite{BS07_III}). 
Although our model falls into the second category, we believe 
that one can analyse simple, but not phenomenological, models in probabilistic 
terms to get a better understanding of the effect that processes at the pore scale 
have in the whole system. In this regard, our work is also related to for example 
\cite{DirrDondl} (and the references therein), where the propagation of fronts through 
a random field of obstacles is analysed in relation to hysteresis. 

In the next section we set up the mathematical model and discern 
the parameter regimes under consideration. We outline the contents of the 
paper at the end of this Introduction.

\subsection{The mathematical model and main results}\label{sec:model}
We consider a two dimensional rectangular container of length $L$ that contains
 $N$ uniformly distributed circular grains. If the average distance between grains is $d>0$ 
(understood as the average of the distances of every 
center to its nearest neighbour), then $N\propto L^2/d^2$. We then let $\nu>0$ represent  
the density of grains, i.e. $\nu$ satisfies $N=\nu L^2/d^2$. 
We assume, for simplicity, that all grains are of equal size and we denote by 
$R<d/2$ their radius.

The container is partially filled with a liquid so that an interface forms 
separating the liquid from the air. We assume that such an interface splits the domain in two regions, 
i.e. it does not intersect itself and there are no isolated bubbles of either 
of the fluid phases. Let us denote by $\Gamma$ the curve that joins grains and forms the 
interface, that is the union of capillary(-pressure) curves $\gamma$ joining 
every two grains that meet the interface and that we call {\it elemental components} of $\Gamma$. 
A sketch of this setting is shown in Figure~\ref{basic:setting}. 
\begin{figure}[hhh]
\begin{center}
\includegraphics[width=0.75\textwidth,height=.35\textwidth]{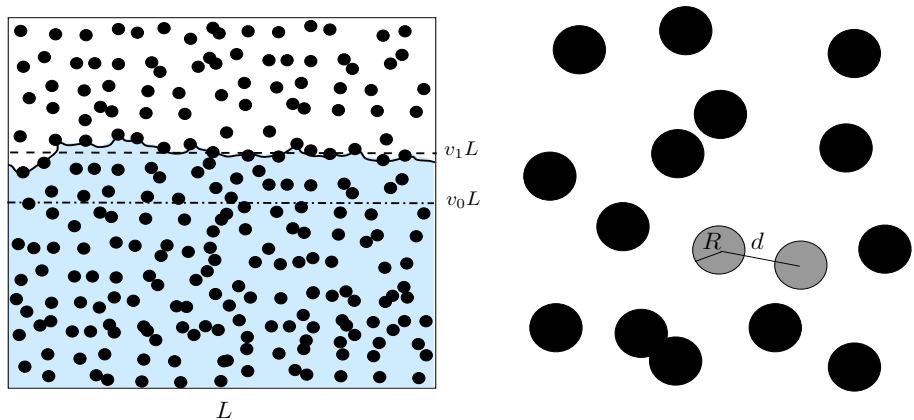}
\end{center}
\caption{Schematic picture showing the problem set-up and the three relevant length scales.
}
\label{basic:setting}%
\end{figure}

Let us fix the canonical basis $\{\mathbf{e}_{1},\mathbf{e}_{2}\}$ such that 
$\mathbf{e}_{1}=(1,0)$ and $\mathbf{e}_{2}=(0,1)$, then $\mathbf{e}_{2}$ 
points in the opposite direction to the gravitational field. We shall denote 
the coordinates of a point $x\in\mathbb{R}^{2}$ by $x=(x_{1},x_{2})$ in this basis.

Every elemental component of $\Gamma$, $\gamma$, satisfies the capillary 
equation (e.g. \cite{Finn})
\[
\sigma H = \rho\, g x_{2} - p\,,
\] 
that must be complemented with Young's condition at the grains and/or the walls
 of the container. We shall assume that the contact angle is the same at all 
solid surfaces. Here $H$ denotes the signed curvature of each $\gamma$ that forms 
$\Gamma$, $\rho$ is the (constant) density of the liquid, $g$ is the gravity 
constant and $\sigma$ denotes surface tension.
Here $p$ is the pressure associated to the volume fraction of the domain 
(including grains), $v_{1}\in (0,1)$, occupied by the liquid. Then 
\[
p = \rho \,g \lambda\quad \mbox{with}\quad \lambda=v_{1}L\,.
\]
We take as unit of length the average distance between the grains of the system, $d$. 
Up-scaling the length variables with $d$, we obtain that $\Gamma$ satisfies the 
non-dimensional equation
\begin{equation}
H=B\,(x_{2}-\lambda) \,. \label{S1E1}%
\end{equation}
subject to a contact angle boundary condition that we specify below. 
We obtain the non-dimensional number $B$, which is the Bond number 
\begin{equation}
B=\frac{\rho g d^{2}}{\sigma}\,, \label{S1E2}%
\end{equation}
and measures the balance between gravitational forces and surface tension in this units of length. 

For simplicity, we use the same notation for the non-dimensional variables. 
Also we let $L$ and $R$ be the non-dimensional size of the domain (i.e. $L/d$) and grain radius (i.e. $R/d$), 
respectively. 
This, in particular, means that, with this notation, $\lambda = v_{1}L$ in (\ref{S1E1}). 
We observe that $v_1$ depends on the configuration of grains, so if
we let $v_{0}\in[0,1]$ be the volume fraction of the domain occupied by the liquid, the volume 
occupied by the liquid is $v_{0}L^{2}$ and then 
\begin{equation}\label{v0:vs:v1}
v_{1}L^{2}\approx v_{0}L^{2}+n\pi R^{2}
\end{equation} 
is the volume below the interface $\Gamma$, if $n$ is the number of grains below $\Gamma$. 

We assume thereafter that $v_0$ is fixed, and we take the horizontal line 
$x_{2}=v_{1} L$($=\lambda$) as a reference height for each configuration 
of grains. Observe that this is a solution of (\ref{S1E1}) with $\alpha=0$ 
at the walls.
 
We also observe that with the new notation, where $L$ and $R$ are non-dimensional lengths, 
then the number of grains in the container is 
$N=\nu L^2$. Henceforth, with this scaling we can think of the equivalent problem where 
the liquid and air occupy some space in a square 
domain $[0,L]^{2}$ with $L\gg 1$, where $\nu L^2$ grains are uniformly distributed 
with an average distance $1$ between them.

 We indicate the position of a center by the variable $\xi=(\xi^{(1)},\xi^{(2)})\in\mathbb{R}^2$
 and write its boundary in the form $\partial B_{R}(\xi)$.
We take $\mathbf{n}$ as the unit normal vector to $\Gamma$ pointing in the direction of air, 
and $\mathbf{t}$ the tangent 
vector to $\Gamma$ such that $\{\mathbf{t},\mathbf{n}\}$ has the 
same orientation as the canonical basis (i.e. that $\det (\mathbf{t},\mathbf{n})=1$).

Let us denote by $\mathbf{N}$ the outer normal vector to $\partial B_{R}(\xi)$ at a point
 $P$ where $\Gamma$ and $\partial B_{R}(\xi)$ intersect (i.e. $P$ is a contact point).
 Let us denote by $\mathbf{t}_{c}$ the value of $\mathbf{t}$ at $P$. Then, the Young condition 
is defined as follows: if $\mathbf{N}\cdot\mathbf{t}_{c}\mathbf{\geq0}$, 
\begin{equation}
\mathbf{N}\wedge\mathbf{t}_{c}=\sin(\alpha)  \mathbf{e}_{3} 
\label{S2E2}%
\end{equation}
where $\alpha\in\left[-\frac{\pi}{2},\frac{\pi}{2}\right]$ is the contact
angle and where $\mathbf{e}_{3}=\mathbf{e}_{1}\wedge\mathbf{e}_{2}$. If
$\mathbf{N}\cdot\mathbf{t}_{c}<0$, then with the same choice of
$\mathbf{e}_{1}$, we set 
$\mathbf{N}\wedge\mathbf{t}_{c}=\sin(\alpha-\pi) \mathbf{e}_{3}$.

We shall distinguish different parameter regimes. The following observations are 
then in order. We recall that $B$ is the Bond number with respect to the scale 
given by the distance between grains. In the case $B\gg 1$ the gravitational effects 
are the dominant ones. In this limit the liquid-gas interfaces tend to be horizontal lines, 
corrected by boundary layers of height $1/\sqrt{B}$. In the limit $B\ll 1$, the interfaces 
are dominated by surface tension and this allows vertical variations in the connection between 
grains (of the order of the distance between the grains). Thus, unlike in the case $B\gg 1$, 
if $B\ll 1$ the connections can take place in any direction, not just horizontally.

The above regimes take into account the size of $B$ only, but we shall relate 
it to $L$ as well. In particular, the limit $L\to \infty$ gives an ever 
increasing number of grains and can in principle give rise to 'anomalous' 
situations (not sketched above) that we analyse in detail in Section~\ref{sec:stoch} and that we summarise below.

We distinguish several cases depending on the relative size of $B$ and $L$. 
In the cases where $B\gg 1$, we actually compare $L$ to 
$\min\{\sqrt{B},R^{-1}\}$, which relates the height of the boundary 
layers to $R$:

\begin{enumerate}
\item {\bf The regime $ 1\lesssim L\ll \min\{\sqrt{B},R^{-1}\}$, where $B\to\infty$ and $R\to 0$ 
(Section~\ref{regime:1} Theorem~\ref{theo:case1}):} 
In this case gravity dominates and the size of the container is not 
sufficiently large to allow significant capillarity effects. We find that 
as $B\to \infty$ there are configurations with probability 
close to $1$ for which the solution interfaces tend to the horizontal solution. 
One can in fact distinguish the cases $L=O(1)$ and $L\to\infty$, with either $R\ll1/\sqrt{B}$ 
or $R\gg 1/\sqrt{B}$, with the same result.

\item {\bf The regime $L=O(\min\{\sqrt{B},R^{-1}\})$ and $L\to\infty$ 
(Section~\ref{regime:2} Theorem~\ref{theo:case2}):} 
In this case we find that the probability of  finding solutions close to the horizontal line is close to one. 
On the other hand we also find that the probability of finding configurations 
that have a given number of elemental components is close to zero, but bounded from above 
by a Poisson distribution. Again, the distinction between the cases $1/\sqrt{B}\gg R$ 
and $1/\sqrt{B}\ll R$ is not significant.

\item  {\bf The regime $L \gg \min\{\sqrt{B},R^{-1}\}$ 
and $L\to \infty$:}

\begin{enumerate}
\item {\bf The case $B\to\infty$ (and $R\to 0$) (Section~\ref{regime:3} Theorem~\ref{theo:case3})}. 
This is a regime where gravity in principle dominates, we find an interval of $L$ in terms of $B$ where 
with probability close to one solutions stay in a narrow strip around the horizontal reference line. 
For $L$ above this range there is no guarantee that solutions are necessarily close to the horizontal
 reference line. If  $1/\sqrt{B}\gg R$, the interval is  $\sqrt{B}\ll L \ll\sqrt{B}\log(B)$, and if 
$1/\sqrt{B}\ll R$ with $R\to 0$, the result holds replacing $1/\sqrt{B}$ by $R$.

 \item {\bf The case $B=O(1)$}. In this case, capillary and gravity forces are comparable and it is 
not possible to derive general properties about the structure of the solution interfaces using 
perturbative methods. Therefore this regime is not included in our analysis.

\item {\bf The case $B\to 0$ (Section~\ref{regime:4} Theorem~\ref{percolation:theo})}. 
We will see that in this case the only meaningful limit is that of $L=O(1/B)$. When $B\to 0$, the maximum 
fluctuation of height of a solution $\Gamma$ are of the order of $1/B$. In this regime, we can show that
 with probability close to one, we can find interfaces that are suitably close to any smooth curve without
 self-intersections that connect (following its orientation) $\{0\}\times[0,L]$ to $\{L\}\times[0,L]$. 
\end{enumerate}
\end{enumerate}

In order to understand the probabilistic problem we need to analyse the elemental 
components that an interface $\Gamma$ can have. This is done in Section~\ref{sec:construction1} 
where we analyse the solution curves of (\ref{S1E1}) and characterise them in 
terms of a free parameter that controls the height reached by the curve and 
its curvature. In fact, equation (\ref{S1E1}) implies that the curvature of an 
elemental component increases in absolute value with the distance to the reference height. 
The main results and their proofs are presented 
in Section~\ref{sec:stoch}, where the regimes outlined above are studied separately. 
Finally, Appendix~\ref{sec:proofgeolemma} deals with the proof of a technical lemma that 
is needed in Section~\ref{regime:3}.

\section{Preliminary Analysis: Gluing of Solutions}
\label{sec:construction1} 
In this section we analyse the basic features of elemental components of an 
interface $\Gamma$. We first observe that equation (\ref{S1E1}) in the absence
 of boundary conditions can be solved explicitly in terms of elliptic integrals, 
see \cite{Myshkis}. One can achieve this by first introducing the variable 
$\bar{x}_{2}=x_{2}-\lambda$ and using the angle $\beta\in\lbrack-\pi,\pi]$ from 
$\mathbf{e}_{1}$ to the tangent vector $\mathbf{t}=(x_{1}^{\prime}(s),\bar{x}_{2}^{\prime}(s))$ 
where $s$ is the arc-length. Then multiplying (\ref{S1E1}) by $\bar{x}_{2}^{\prime}(s)$ 
and integrating once we get
\begin{equation}
B\frac{\bar{x}_{2}^{2}}{2}+\cos\beta=C\,, \label{A1E1}%
\end{equation}
where $C$ is a constant of integration that can have either sign or be zero,
but clearly $C\geq-1$. Then solving (\ref{S1E1}), with given initial
conditions, amounts to solving
\begin{equation}
x_{1}^{\prime}(s)=\cos\beta\quad\bar{x}_{2}^{\prime}(s)=\sin\beta
\label{t:beta:rela}%
\end{equation}
subject to the same initial conditions, together with (\ref{A1E1}) and a
compatible constant $C$. 

The are the obvious symmetries of equation
(\ref{S1E1}), $(x_{1},s)\rightarrow(2a-x_{1},-s)$ (for any $a\in\mathbb{R}$)
and $\bar{x}_{2}\rightarrow-\bar{x}_{2}$, that will be used to simplify the
computations below.

Using (\ref{t:beta:rela}) and the definition of arc-length, one can write
(\ref{A1E1}) as
\begin{equation}
\left[  1-\left(  C-B\frac{\bar{x}_{2}^{2}}{2}\right)  ^{2}\right]  \left(
dx_{1}\right)  ^{2}=\left(  C-B\frac{\bar{x}_{2}^{2}}{2}\right)  ^{2}\left(
d\bar{x}_{2}\right)  ^{2} \label{A1E2}%
\end{equation}
The fact that we are taking squares to go from (\ref{S1E1}) to (\ref{A1E2})
means that the set of solutions of the latter is larger than that of the
former. In the construction one has to impose that the resulting curves are
indeed solutions of (\ref{S1E1}), as we shall see. We also notice that all
solutions must have $|\bar{x}_{2}|\leq  \sqrt{2(C+1)/B}$. 

Below, we find the solutions of
(\ref{S1E1}) by gluing together pieces of curve solutions that are graphs of
a function of the form $x_{1}(\bar{x}_{2})$ and that solve (\ref{A1E2}). 
We then integrate (\ref{A1E2})
after taking the positive or negative square root of $(dx_{1})^{2}/(d\bar
{x}_{2})^{2}$; we use the equations:%
\begin{equation}
\frac{dx_{1}}{d\bar{x}_{2}}=
\pm
\frac{
C - B \displaystyle{ \frac{\bar{x}_{2}^{2}}{2} } 
}%
{
\sqrt{ 1 - \left(  C - B \displaystyle{ \frac{\bar{x}_{2}^{2}}{2} } \right)^{2} }
}\,.
\label{Dx1:Dx2}%
\end{equation}
Observe that $\beta$, for such graphs, is characterised by
\[
\frac{dx_{1}}{d\bar{x}_{2}}=\frac{1}{\tan\beta}\,.
\]

There are solutions, however, that cannot be constructed the way explained
above. These are the solutions with $\bar{x}_{2}\equiv0$, that is the curves
$(s,0)$ (air is on top of the liquid) and $(-s,0)$ (air is below the
liquid). The former corresponds to a solution of (\ref{A1E1}) with $C=1$, and
the latter to a solution of (\ref{A1E1}) with $C=-1$.

Next, we identify the points of the solution curves where either $x_{1}^{\prime}(s)=0$ or 
$\bar{x}_{2}^{\prime}(s)=0$, i.e. the values of $\bar{x}_{2}$ where either the left or the 
right-hand side of (\ref{A1E2}) vanishes. For simplicity, we use the following notation:
\begin{equation}\label{x2:mim}
x_2^{min}(C;B):= 
\left\{
\begin{array}{l}
\displaystyle{ \sqrt{\frac{2(C-1)}{B}}} \quad  \mbox{if} \quad C>1 \\
0 \quad  \mbox{otherwise} 
\end{array}
\right.
\end{equation}
and 
\begin{equation}\label{x2:med:min} 
x_2^{med}(C;B) := \sqrt{\frac{2C}{B}}\,, \quad 
x_2^{max}(C;B):=\sqrt{\frac{2(C+1)}{B}}
\end{equation}

\begin{lem}\label{hori:vert:points} 
The points at which either $x_{1}^{\prime}(s)=0$ or
$\bar{x}_{2}^{\prime}(s)=0$ are as follows:
\begin{enumerate}
\item If $C\geq0$, when $x_{1}^{\prime}(s)=0$ then $\bar
{x}_{2}=\pm x_2^{med}(C;B)$, and when $x_{2}^{\prime}(s)=0$ then either
$\bar{x}_{2}=\pm x_2^{min}(C;B)$ or $\bar{x}_{2}=\pm x_2^{max}(C;B)$,  
except if $0\leq C<1$, that $x_{2}^{\prime}(s)=0$
only occurs at $\bar{x}_{2}=\pm x_2^{max}(C;B)$.

\item If $C<0$, then there is no value of $s$ such that $x_{1}^{\prime}(s)=0$, and 
as before, $x_{2}^{\prime}(s)=0$ occurs at $\bar{x}_{2}=\pm x_2^{max}(C;B)$.
\end{enumerate}
\end{lem}

The next lemma describes the way solutions of (\ref{S1E1}) can be obtained by
gluing solutions of (the branch equations) (\ref{Dx1:Dx2}) at points with
either vertical or horizontal tangent vector.

\begin{lem}[Gluing rule]\label{glue:rule} For $C>-1$, given the initial condition 
$(x_{1}(0),\bar{x}_{2}(0))=(a, x_2^{max}(C;B))$ with $a\in\mathbb{R}$ and
$(x_{1}^{\prime}(0),\bar{x}_{2}^{\prime}(0))=(-1,0)$ for (\ref{S1E1}), this
specifies the branch equation of (\ref{Dx1:Dx2}) that the corresponding
solution (\ref{S1E1}) satisfies. If this solution reaches a point with
$x_{1}^{\prime}(s)=0$ then it can only be continued further by solving the
same branch of (\ref{Dx1:Dx2}). On the other hand, if the curve reaches a
point with $\bar{x}_{2}^{\prime}(s)=0$ then it can only be continued further
by solving the branch equation with opposite sign.
\end{lem}

\begin{proof}
A solution stops being a graph of a function of the form $x_{1}(\bar{x}_{2})$
when, seen as a curve, $\bar{x}_{2}^{\prime}(s)=0$. It is easy to see that,
this never occurs at a point with $\bar{x}_{2}=0$. At those points we can
further continue the solution curve by gluing smoothly with another solution
curve that solves (\ref{Dx1:Dx2}) with the opposite sign (notice that these
curves, as graphs, have the opposite orientation, and hence the change of sign
in (\ref{Dx1:Dx2}) implies that they have the same curvature at that point).

On the other hand, at points with $x_{1}^{\prime}(s)=0$, the solution is still
a graph, and in principle, one could glue it to a solution that solves
(\ref{Dx1:Dx2}) with the opposite sign. However, if this happens at a point
with $\bar{x}_{2}\neq0$, then $H$ would change sign, which contradicts
(\ref{S1E1}). If, this happens at a point with $\bar{x}_{2}=0$, then $H$ keeps
its sign, but this also contradicts (\ref{S1E1}).
\end{proof}

We now construct pieces of curve solutions from which one can recover complete 
solutions of (\ref{S1E1}) (i.e. defined for all $s\in \mathbb{R}$) by the symmetries 
$(x_{1},s)\to(2a-x_{1},-s)$, for any $a\in\mathbb{R}$, and $\bar{x}_{2}\to-\bar{x}_{2}$. 
The next lemma gives these pieces for every $C>-1$ by distinguishing the cases $C>1$, 
$1>C>0$ and $0\geq C>-1$.

\begin{lem}[The pieces]\label{funda:piece} 
Given the initial condition $(x_{1}(0),\bar{x}_{2}(0))=(a,x_2^{max}(C;B)$, 
with $C>-1$ and $a\in\mathbb{R}$, and
$(x_{1}^{\prime}(0),\bar{x}_{2}^{\prime}(0))=(-1,0)$ for (\ref{S1E1}), the
elemental component of the solution curve is obtained by integrating (\ref{Dx1:Dx2}), 
with the minus sign, 
up to a $\bar{x}_{2}=0$ or up to $\bar{x}_{2}=x_2^{min}(C;B)$. Namely,

\begin{enumerate}
\item For $C>1$, the curve starts with $dx_{1}/d\bar{x}_{2}\geq0$ and reaches
$\bar{x}_{2}=x_2^{med}(C;B)$ for some value of $s$, for which
$x_{1}^{\prime}(s)=0$ and $\bar{x}_{2}^{\prime}(s)=-1$. The curve continues
with $dx_{1}/d\bar{x}_{2}\leq0$ $\ $up to a point where $\bar{x}_{2}%
=x_2^{min}(C;B)>0$ for some other value of $s$, for which
$x_{1}^{\prime}(s)=1$ and $\bar{x}_{2}^{\prime}(s)=0$.

\item For $C=1$, the curve starts with $dx_{1}/d\bar{x}_{2}\geq0$ and reaches
$\bar{x}_{2}=\sqrt{\frac{2}{B}}$, for which $x_{1}^{\prime}(s)=0$ and $\bar
{x}_{2}^{\prime}(s)=-1$. The curve continues with $dx_{1}/d\bar{x}_{2}\leq0$
and has an asymptote at $\bar{x}_{2}=0$, that is with $x_{1}(s)\to
\infty$ and $\bar{x}_{2}(s)\to 0^{+}$ as $s\to\infty$.

\item For $0<C<1$, the curve starts with $dx_{1}/d\bar{x}_{2}\geq 0$ and
reaches $\bar{x}_{2}=x_2^{med}(C;B)$ at some value of $s$, for which
$x_{1}^{\prime}(s)=0$ and $\bar{x}_{2}^{\prime}(s)=-1$. The curve continues
with $dx_{1}/d\bar{x}_{2}\leq0$ $\ $up to a point where $\bar{x}_{2}=0$ at
some other value of $s$, for which $x_{1}^{\prime}(s)>0$ and $\bar{x}_{2}^{\prime}(s)<0$.

\item For $-1<C\leq 0$, the curve starts with $dx_{1}/d\bar{x}_{2}\geq 0$ and
reaches $\bar{x}_{2}=0$ at some value of $s$, for which $x_{1}^{\prime}(s)=0$ and
$\bar{x}_{2}^{\prime}(s)=-1$, if $C=0$, and $x_{1}^{\prime}(s)<0$ and $\bar{x}_{2}^{\prime}(s)<0$, 
otherwise.
\end{enumerate}
\end{lem}

\begin{proof}
We only prove {\it (i)}, thus, let us assume that $C>1$. 
We first observe that due to the fact that the
maximum value that $\bar{x}_{2}$ can take is the initial one, the direction of
the initial tangent vector implies that away from, but close to, the initial
point $dx_{1}/d\bar{x}_{2}>0$. Then the curve solution can be seen as a graph
of a function of the form $x_{1}(\bar{x}_{2})$ that solves (\ref{Dx1:Dx2})
with the minus sign. Then the solution $x_{1}(\bar{x}_{2})$ is convex as long
as $\bar{x}_{2}\in(x_2^{min}(C;B), x_2^{max}(C;B))$ (see (\ref{S1E1}), that implies 
$\mbox{sign}(d^{2}x_{1}/d\bar{x}_{2}^{2})=\mbox{sign}(\bar{x}_{2})$). In particular, 
this curve reaches a point with $\bar{x}_{2}=x_2^{med}(C;B)$ (i.e. with 
$(x_{1}^{\prime}(s),\bar{x}_{2}^{\prime}(s))=(0,-1)$ for some value of $s$). In order to prove 
that the curve also reaches a point with $\bar{x}_{2}=x_2^{min}(C;B)$ (i.e. with 
$(x_{1}^{\prime}(s),\bar{x}_{2}^{\prime}(s))=(1,0)$ for some value of $s$), we construct 
a curve ending up at such a point, and by a similar argument this curve also passes 
through a point with $\bar{x}_{2} = x_2^{med}(C;B)$. Then, adjusting the initial value of 
$x_{1}$ so that both curves intersect at such a point (which they do tangentially), 
and using the uniqueness provided by Lemma~\ref{glue:rule} imply that they are the same solution curve.

All the other cases are straight forward by convexity of the solution curve.
\end{proof}

We show in figures~\ref{Cbigger1} (for $C>1$),~\ref{Cless1pos} (for $0<C<1$), ~\ref{Ce1} 
(for $C=1$), ~\ref{Ce0} (for $C=0$) and ~\ref{Cng} (for $-1<C<0$) the pieces described in Lemma~\ref{funda:piece}.
The next lemma describes how complete curves are obtained from these pieces.

\begin{figure}[hhh] 
\centering
\mbox{\subfigure[$C>1$]
{
\includegraphics[width=0.35\textwidth,height=.25\textwidth]{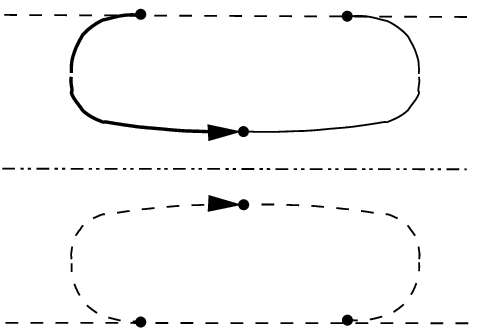}
\label{Cbigger1}
}
} 
\mbox{
\subfigure[$0<C<1$]
{
\includegraphics[width=0.35\textwidth,height=.25\textwidth]{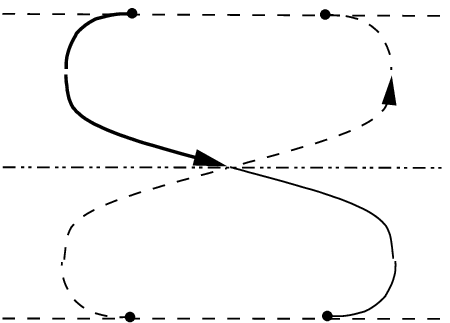}
\label{Cless1pos}
}
}\\
\mbox{
\subfigure[$C=0$]
{
\includegraphics[width=0.35\textwidth,height=.25\textwidth]{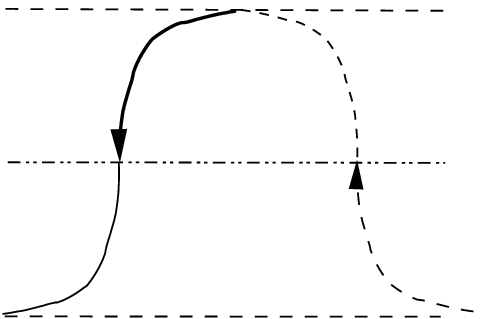}
\label{Ce0}
}
}
\mbox{
\subfigure[$-1<C<0$]
{
\includegraphics[width=0.35\textwidth,height=.25\textwidth]{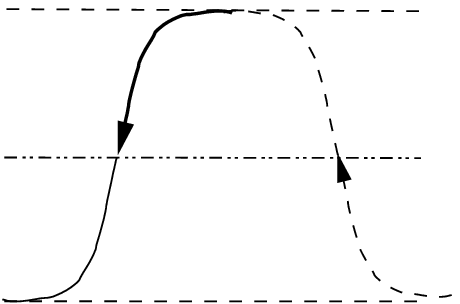}
\label{Cng}
}
}
\caption{This picture shows solution curves of (\ref{A1E1}) with $C>0$ and
$C\neq1$. The dashed-dotted line represents the line $\bar{x}_{2}= 0$, and
the dashed horizontal ones are the lines $\bar{x}_{2}=\pm x_2^{max}(C;B)$. The
thick solid lines represent the pieces of curves as specified in
Lemma~\ref{funda:piece}. The remaining lines give solutions obtained by the
symmetries $\bar{x}_{2}\to-\bar{x}_{2}$ and $x_{1}\to-x_{1}$, see Lemma~\ref{complete:curves}. 
Any other solution is a translation of the ones shown. Complete solutions curves to (\ref{S1E1}) 
are achieved by the smooth gluing of these ones at the points with $\bar{x}_{2}^{\prime}(s)=0$ on the dashed lines. 
\subref{Cbigger1} shows the case $C>1$ and \subref{Cless1pos} shows the case in which $0<C<1$. 
According to the orientation shown in these cases the air lies `on top' of the graphs shown. 
\subref{Ce0} shows the case $C=01$ and \subref{Cng} shows the case in which $-1<C<0$. 
In these cases, the air lies below the graphs.
}%
\label{Cneq1pos}%
\end{figure}

\begin{lem}[Complete curve solutions]\label{complete:curves} For each $C>-1$, 
the pieces of solution curves obtained in Lemma~\ref{funda:piece} give rise to complete curves, as follows:

\begin{enumerate}
\item There are three additional pieces of curve solutions obtained by the symmetries $x_{1}\to -x_{1}$ 
and $\bar{x}_{2}\to -\bar{x}_{2}$.


\item In all cases above the curves obtained can be glued together smoothly, translating 
as necessary in $x_{1}$, at the points where $\bar{x}_{2}^{\prime}(s)=0$, and at points with 
$\bar{x}_{2}=0$ (choosing pieces with the same tangent vector at the point).
\end{enumerate}
\end{lem}

\begin{rem}\label{upside:down}
\begin{enumerate}
\item The complete curves that we obtain are open and periodic except for a particular 
value of $C$ for which the complete curves are closed: by continuity, there exists a value 
of $C\in(0,1)$, say $C_{0}$ (actually $C_{0}\approx 0.6522$, see \cite{Myshkis}), such that 
the corresponding elemental component satisfies that 
$x_{1}(0)=x_{1}\left( x_2^{max}(C;B) \right)$. In particular, if $C>C_{0}$, the curve obtained 
in Lemma~\ref{complete:curves} has, globally, increasing $x_{1}$ as $s\to\infty$, whereas 
the corresponding curve if $C<C_{0}$ has, globally, decreasing $x_{1}$ as $s\to\infty$. 

Physically, one can interpret the solution curves with $C<C_{0}$ as describing a situation 
in which air lies below the liquid, and {\it vice versa} for the ones with $C>C_{0}$. 

\item It is also worth mentioning that the complete curve obtained for $C\in(-1,0]$ 
that oscillates around $\bar{x}_{2}=0$ does not have self-intersections and can in 
principle connect two grains very far apart (in the horizontal direction).   
\end{enumerate}
\end{rem}

The next lemma gives the range of $\beta$ that complete curve have depending on $C$:
\begin{lem}[Range of $\beta$]\label{beta:values} For the complete curves obtained in 
Lemma~\ref{complete:curves} the values of the angle $\beta$ range in
different intervals depending on $C$:

\begin{enumerate}
\item If $C>1$, $\beta$ takes all values in $[-\pi,\pi]$.

\item If $C=1$, $\beta$ takes all values in $[-\pi,0)\cup(0,\pi]$, unless the
solution is the horizontal line, for which $\beta\equiv0$.

\item If $-1\leq C<1$, $\beta$ takes all values in $[-\pi,-\eta]\cup
\lbrack\eta,\pi]$ for some $\eta\in(0,\pi]$, and $\eta$ increases when
$C$ decreases (in fact, $\eta=\pi$ if and only if $C=-1$, i.e. $\bar{x}_{2}\equiv0$ with $\mathbf{n}$ 
pointing in the direction of gravity, and $\eta=\pi/2$ if and only if $C=0$).
\end{enumerate}
\end{lem}

\begin{rem}\label{beta:range}
\begin{enumerate}
\item We observe that by the determination of $\beta$ that we have chosen, its
value jumps from $\pi$ to $-\pi$ at gluing points with 
$\bar{x}_{2}=x_2^{max}(C;B)$, and, similarly, its value jumps from $-\pi$ to 
$\pi$ at gluing points with $\bar{x}_{2}=-x_2^{max}(C;B)$.

\item We observe that given the contact angle $\alpha$ and the radius of grains $R$,
the contact point at a grain is given by the angle $\rho_c$ measured from the horizontal placed 
at the center of the grain to the contact point, 
and satisfies
\begin{equation}\label{beta:rho}
\rho_c = \alpha - \beta_c
\end{equation} 
where $\beta_c$ is $\beta$ at the contact point.

Then Lemma~\ref{beta:values} gives the possible range of $\rho_c$ for a given $\alpha$· 
In particular, if $C\geq 1$ $\rho_c$ covers one period 
and unique solution exists between two given grains that are sufficiently close by adjusting $C$.   
If $-1\leq C<1$, then $\rho_c\in[\alpha-\pi,\alpha-\eta]\cup[\alpha+\eta,\alpha+\pi]$. This means 
that there are some restrictions on the contact surface $\partial B_R(\xi)$ for grains that lie 
near the horizontal if air lies below the interface.  

\end{enumerate}
\end{rem}

We give the solutions of (\ref{A1E1}) with $C=1$ explicitly in the next
lemma, see also Figure~\ref{Ce1} 
\begin{figure}[hhh]
\begin{center}
\mbox{
\includegraphics[width=0.40\textwidth,height=.22\textwidth]{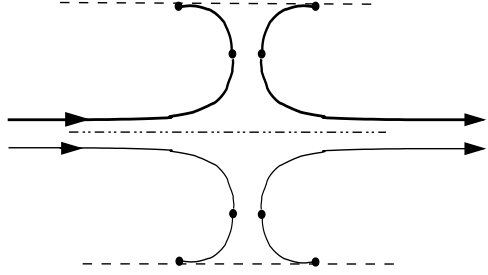}
}
\end{center}
\caption{This picture shows four solution branches of (\ref{A1E1}) with $C=1$.
One if then is, in this case, $\bar{x}_{2}\equiv 0$ (dashed-dotted line). 
The complete solution curves to (\ref{S1E1}) are achieved by smoothly gluing at $\bar{x}_{2}=x_2^{med}(C;B)$ 
or at $\bar{x}_{2}=-x_2^{med}(C;B)$ (dashed line), depending on the sign, and translating in the direction of 
$x_1$ copies of the ones shown. }
\label{Ce1}
\end{figure}

\begin{lem}\label{lemC1} The solutions of (\ref{A1E1}) with $C=1$ are explicitly given by
the horizontal line $\bar{x}_{2}\equiv0$ oriented with $\mathbf{n}%
=\mathbf{e}_{2}$, and by two one-parameter families of solution branches:%
\[
x_{1}^{+}(\bar{x}_{2};a)=-\frac{1}{\sqrt{B}}F(\sqrt{B}\bar{x}_{2}%
)+a,\;\;\bar{x}_{2}\in\left(  -\frac{2}{\sqrt{B}},0\right)  \cup\left(
0,\frac{2}{\sqrt{B}}\right)  ,
\]%
\[
x_{1}^{-}(\bar{x}_{2};a)=\frac{1}{\sqrt{B}}F(\sqrt{B}\bar{x}_{2}%
)+a,\;\;\bar{x}_{2}\in\left(  -\frac{2}{\sqrt{B}},0\right)  \cup\left(
0,\frac{2}{\sqrt{B}}\right)  \,,
\]
where the parameter $a$ is a real number and $F$ is given by%
\begin{equation}\label{F:def}
F(Z)=\frac{1}{2}\log\left(  \frac{2-\sqrt{4-Z^{2}}}{2+\sqrt{4-Z^{2}}}\right)
+\sqrt{4-Z^{2}},\;Z\in(-2,0)\cup(0,2)\,. 
\end{equation}

Moreover, $F$ satisfies the following properties: $\lim_{Z\to-2^{+}}F(Z)=\lim_{Z\to 2^{-}}F(Z)=0$ 
and $\lim_{Z\to 0^{\pm}}F(Z)=-\infty$. Observe also that $F^{\prime}(Z)=0$ at $Z=\pm\sqrt{2}$ and
that $\lim_{Z\to -2^{+}}F^{\prime}(Z)=\infty$, $\lim_{Z\to 2^{-}}F^{\prime}(Z)=-\infty$ and 
$\lim_{Z\to 0^{\pm}}F^{\prime}(Z)=\pm\infty$. Regarding convexity; $F^{\prime\prime}(Z)<0$ with
$Z\in(-2,0)\cup(0,2)$. 
\end{lem}

\begin{rem}\label{C1:complete}
\begin{enumerate}
\item We notice that the functions $x_{1}^{+}$ and $x_{1}^{-}$ can be extended to
$\bar{x}_{2}=\pm 2/\sqrt{B}$, with $x_{1}^{+}(\pm2/\sqrt{B};a)=a$ and
$x_{1}^{-}(\pm2/\sqrt{B};a)=a$, but the derivative at those points is not
finite. We also observe that gluing together the branch $x_{1}^{+}$ and
$x_{1}^{-}$ with $\bar{x}_{2}\in[-2/\sqrt{B},0)$, and the same value of
$a$ we obtain a complete planar curve. Similarly, $x_{1}^{+}$ and $x_{1}^{-}$
with $\bar{x}_{2}\in(0,2/\sqrt{B}]$ for the same $a$ give another complete
curve, that can be simply recovered from the previous one by the
transformation $\bar{x}_{2}\to -\bar{x}_{2}$.

\item We observe that the limit $C\to 1$ can be studied independently of 
the value of $B$. One can easily show that solutions with positive $C$, 
and either $C>1$ or $C<1$ converge to the solutions found in Lemma~\ref{lemC1} 
on compact sets of $x_2$. 
\end{enumerate}
\end{rem}

The rest of the section is devoted to deriving estimates on the maximal horizontal length of 
a solution curve depending on the value of $C$. 
First we derive the horizontal length at points where the height is minimum or $0$ to 
points where the height is $x_2^{med}$ or $x_2^{max}$: 
\begin{lem}\label{C:estimates}
Given a elemental component solution as in Lemma~\ref{funda:piece}:
\begin{enumerate}
\item If $-1<C\leq 0$, then
\[    
\log\left(\frac{2+\sqrt{2}}{1+\sqrt{3}}\right)\frac{1}{\sqrt{B}}
<
x_{1}( x_2^{max}(C;B)) - x_{1}(0)
<
\frac{\pi}{\sqrt{2B}}\,.
\]

\item If $0<C<1$, then
\[
\frac{1}{\sqrt{B}}h_{1}(C)
<
x_{1}(0)-x_{1}\left( x_2^{med}(C;B)\right)
<
\frac{1}{\sqrt{B}}h_{2}(C)\,.
\]
where $h_{i}(C)\to 0$ as $C\to 0$ and $h_{i}(C)=O(-\log(\sqrt{1-C}))$ as $C\to 1^{-}$, for both $i=1,2$.

\item If $C>1$, then there exist non-negative functions $h_{3}(C)$ and $h_{4}(C)$,
continuous in $C$, such that, 
\[
\frac{1}{\sqrt{B}}h_{3}(C)
<
x_{1}\left(x_2^{min}(C;B)\right)-x_{1}\left( x_2^{med}(C;B)\right)  
<
\frac{1}{\sqrt{B}}h_{4}(C)\,,
\]
where $h_{i}(C)=O(-\log\sqrt{C-1}))$ as $C\to 1^{+}$,
for both $i=3,4$ and $h_{i}(C)= O(1/\sqrt{C})$ as $C\to \infty$, for both $i=3,4$. 

\end{enumerate}
\end{lem}
\begin{proof}
In order to prove the estimates we use Lemma~\ref{funda:piece} and integrate
(\ref{Dx1:Dx2}) over $\bar{x}_{2}$ in the appropriate interval where
$x_{1}(\bar{x}_{2})$ is monotone. Then we estimate the integrand (that is the
right-hand side of (\ref{Dx1:Dx2})) with the minus sign by using
straightforward bounds on the numerator and on the denominator. Namely, we
estimate the parabola $C-B\frac{Z^{2}}{2}$ from above and below. We also write
$\left(1 - \left(  C - B\frac{Z^{2}}{2}\right)^{2}\right)^{\frac{1}{2}}
=\left(1 - \left(  C - B\frac{Z^{2}}{2}\right)  \right)^{\frac{1}{2}} 
\left(1 + \left(  C - B\frac{Z^{2}}{2}\right)  \right)  ^{\frac{1}{2}} $. 
Then we either 
estimate the first or the second factor using the bounds on the parabola. The
resulting integrals are elementary. We outline the proof of {\it (i)}, the 
rest is left to the reader.

\textit{(i)} Let us assume first that $-1<C\leq0$. Then
\[
0
<
x_{1}\left(  \sqrt{\frac{2(C+1)}{B}}\right)  - x_{1}(0) = 
- \int_{0}^{\sqrt{\frac{2(C+1)}{B}}} \frac{C-B\frac{Z^{2}}{2}}{\left(  1-\left(
B\frac{Z^{2}}{2} -C \right)  ^{2}\right)  ^{\frac{1}{2}}}dZ \,,
\]
therefore
\begin{align*}
x_{1}\left(  \sqrt{\frac{2(C+1)}{B}}\right)  - x_{1}(0)\leq\frac{1}{\sqrt
{1-C}} \int_{0}^{\sqrt{\frac{2(C+1)}{B} } } \frac{dZ}{\left(  1+
C-B\frac{Z^{2}}{2} \right)  ^{\frac{1}{2}}}= \frac{\pi}{\sqrt{2B(1-C)}} \,,
\end{align*}
and this shows the estimate on the right-hand side. 

On the other hand,
\begin{align*}
x_{1}\left(  \sqrt{\frac{2(C+1)}{B}}\right)  - x_{1}(0)\geq x_{1}\left(
\sqrt{\frac{2(C+1)}{B}}\right)  - x_{1}\left(  \sqrt{\frac{C+1}{B}}\right)  \,
\geq\,\\
\\
\frac{1-C}{\sqrt{2(C+1)}} \int_{\sqrt{\frac{C+1}{B} } }^{\sqrt{\frac
{2(C+1)}{B}}} \frac{dZ}{\left(  1- C+B\frac{Z^{2}}{2}\right)  ^{\frac{1}{2}}}=
\frac{1-C}{\sqrt{B(C+1)}} \log\left(  \frac{\sqrt{2(C+1) } + 2 }{\sqrt{C+1} +
\sqrt{3-C}} \right)  >0
\end{align*}
One can check that the right-hand tends to $(\sqrt{2}-1)/\sqrt{B}$ as $C\to
-1$, this implies continuity of this term for $C\in[-1,0]$. Its minimum there is 
achieved at $C=0$, computing its value gives the result.
\end{proof}

We now give the estimate on the horizontal length from points at height $x_2^{med}$ to points at height $x_2^{max}$.

\begin{lem}\label{C:estimates:2}
For any $C>0$ the following holds
\[
\frac{1}{\sqrt{B}} h_{5}(C)   \leq 
x_{1}\left(  x_{2}^{max}(C;B)\right)  -x_{1}\left(x_{2}^{med}(C;B)\right)  
\leq\frac{1}{\sqrt{B}} h_{6}(C)  
\]
where $h_{5}(C)$ and $h_{6}(C)$ decrease for $C\in[1,+\infty)$, 
$h_{5}(C)$, $h_{6}(C) \to 0$ as $C\to +\infty$ and $h_5(C)$ tends to a positive 
constant as $C\to 0$. 
\end{lem}

We now adapt these estimates to actual curves that solve (\ref{S1E1}) subject to a Young condition on the grains 
connected. Before that we distinguish two types of solution curves. 
\begin{defin}\label{def:lr:rl}
Let $p_{c}$ and $q_{c}\in \mathbb{R}^2$ be the first and second, following 
orientation, contact points of a 
solution curve $\gamma$. We say that $\gamma$ goes from left to write if 
$p_{c}^{(1)}\leq q_{c}^{(1)}$ and write 
$\gamma_{lr}$. Similarly, we say that $\gamma$ goes from right to left if 
$p_{c}^{(1)}\geq q_{c}^{(1)}$, and we write $\gamma_{rl}$.

We also denote by $\beta(p_{c})$ the angle that $\mathbf{t}$ forms with the horizontal 
at the contact point $p_{c}\in \mathbb{R}^{2}$. With this we define the parameter
\begin{equation}\label{dist:below}
d(v_{1},R,C,p_{c},q_{c}):=2 R + \frac{|\cos\beta(p_{c}) -\cos\beta(q_{c}) |}{(2B(C+1))^{\frac{1}{2}}}>0\,.
\end{equation}

Moreover, we shall indicate with a super-index $+$ that $\gamma$ lies above $x_2=v_1 L$ 
(since it then has positive curvature) and with a super-index $-$ that $\gamma$ lies below 
$x_2=v_1 L$ (negative curvature). When $\gamma$ crosses $x_2=v_1 L$ (its curvature changes 
sign across $x_2=v_1 L$) we indicate it with the super-index $c$. 
\end{defin}

\begin{lem}\label{single:distance}
Let $\xi$ and $\zeta$ be centers of two grains connected by a solution $\gamma$ of 
(\ref{S1E1}) with a Young condition at contact points $p_c$ and $q_c$ (following orientation) 
and that solves (\ref{A1E1}) with $C>1$, then:
\begin{enumerate}
\item If $\gamma=\gamma_{lr}$,
\begin{equation}\label{centers:above}
d(v_{1},R,C,p_{c},q_{c})<\|\xi-\zeta\|< 2 R +2  \left(\frac{5}{B(C-1)}\right)^{\frac{1}{2}}
\end{equation}
\item If $\gamma=\gamma_{rl}$,
\begin{equation}\label{centers:above2}
d(v_{1},R,C,p_{c},q_{c})<\|\xi-\zeta\|< 2 R + \left(\frac{\pi^{2} + 2}{ B(C+1) }\right)^{\frac{1}{2}}\,.
\end{equation}
\end{enumerate}
\end{lem}
\begin{proof}
We write $D:=\|\xi-\zeta\|$ (euclidean norm), then clearly 
$D=2R+\|p_{c}-q_{c}\|$ and we can write:
\[
\|p_{c}-q_{c}\|^{2}= |q_{c}^{(1)}-q_{c}^{(1)} |^{2} + |p_{c}^{(2)}-q_{c}^{(2)}|^{2}
\]
and 
\[
\|p_{c}-q_{c}\|\geq |p_{c}^{(2)}-q_{c}^{(2)} |\,.
\]
Using (\ref{A1E1}) then
\[
B\frac{(p^{(2)}_{c})^{2}}{2} -\frac{(p^{(2)}_{c})^{2}}{2}= \cos\beta(q_{c}) - \cos\beta(p_{c})
\]
which gives the right-hand side of (\ref{centers:above}) and (\ref{centers:above2}). 

On the other hand if $\gamma=\gamma_{lr}$ then 
\[
0 \leq |p_{c}^{(1)}-q_{c}^{(1)} | \leq 2 |x_{1}( x_{2}^{med}(C;B) ) -x_{1}( x_{2}^{min}(C;B)  )|
\]
and \[
|p_{c}^{(2)}-q_{c}^{(2)}|\leq x_{2}^{max}(C;B) - x_{2}^{min}(C;B)=\frac{\sqrt{2}}{\sqrt{B}}\frac{2}{\sqrt{C+1}+\sqrt{C-1}}
\]
If $\gamma=\gamma_{rl}$ then
\[
0 \leq |p_{c}^{(1)}-q_{c}^{(1)} | \leq 2 |x_{1}( x_{2}^{max}(C;B) ) -x_{1}( x_{2}^{med}(C;B)  )|
\]
and
\[
|p_{c}^{(2)}-q_{c}^{(2)}|\leq \frac{\sqrt{2}}{\sqrt{B}}\frac{2}{\sqrt{C+1}+\sqrt{C-1}}
\]
The rest of the proof follows from lemmas~\ref{C:estimates} and \ref{C:estimates:2} and the fact that 
$-\log|1-C|\ll1/\sqrt{|1-C|}$ as $C\to 1$. 
\end{proof}

The following lemmas are proved in a similar way than Lemma~\ref{single:distance}. 
\begin{lem}\label{single:distance2}
Let $\xi$ and $\zeta$ be centers of two grains connected by a solution $\gamma$ of (\ref{S1E1})
 with a Young condition at contact points $p_c$ and $q_c$ (following orientation) and that solves (\ref{A1E1}) $0<C<1$:
\begin{enumerate}
\item If $\gamma=\gamma_{lr}$, then
\begin{equation}\label{lr:Cg0}
d(v_{1},R,C,p_{c},q_{c})
< \|\xi-\zeta\| <  2 R +\frac{6}{(B(1-C))^{\frac{1}{2}}}\,.
\end{equation}

\item If $\gamma=\gamma_{rl}$, then 
\begin{equation}\label{rl:Cg0}
d(v_{1},R,C,p_{c},q_{c})
<\|\xi-\zeta\|< 2 R +\left( \frac{  \pi^{2} + 2 }{B(C+1)}  \right)^{\frac{1}{2}}\,.
\end{equation}
\end{enumerate}
\end{lem}


We recall that for $-1<C\leq 0$ the complete curve does not intersect, so the estimates in 
Lemma~\ref{complete:curves} have to be summed up depending on the number of pieces of the curve 
(see Remark~\ref{upside:down}{\it(ii)}):
\begin{lem}\label{single:distance3}
Let $\xi$ and $\zeta$ be centers of two grains connected by a solution $\gamma$ of (\ref{S1E1}) with a Young 
condition and that solves (\ref{A1E1}) with
$-1<C\leq 0$, then it is necessarily of the form $\gamma_{rl}$ and satisfies
\begin{equation}\label{lr:Cl0}
d(v_{1},R,C,p_{c},q_{c})
<
\|\xi-\zeta\|< 2 R + n\, \left( \frac{2(\pi^{2} + 4)}{B}\right)^{\frac{1}{2}}\,,
\end{equation}
where $n$ is the number of elemental components of $\gamma_{rl}$. 
\end{lem}

\section{The Probabilistic Analysis}\label{sec:stoch} 
In this section we shall give the main results and prove some of them. 
Before we do this, we complete the problem setting by describing its probabilistic 
features in detail.

We observe first that the probability of having at least $m$ centers in 
$V\subset[0,L]^{2}$, $P(V)$, is computed by
\begin{equation}\label{m:centers}
P(V)=\left(  \frac{1}{L^{2}}\int_{V}d^{2}\xi\right)^{m}\,.
\end{equation}

We shall denote by $\Omega_{\nu}(L)$ the set of all possible configurations of
$\nu L^{2}$ centers in $[0,L]^{2}$ that are distributed homogeneously and 
independently. In particular, there are $\nu L^{2}$ centers in $[0,L]^{2}$, 
and we can make the identification $\Omega_\nu(L)=\left(  [0,L]^{2}\right)^{\nu L^2}$. 
We assume that the grains have the same size, where their radius $R<1/2$. 
Notice that we allow situations of overlapping grains. Then, we assign the 
following probability measure to $\Omega_\nu(L)$:
\begin{equation}
\mu_{\nu}(d\xi) =\frac{1}{L^{2\nu L^{2}}}\prod_{k=1}^{\nu L^{2}}  d\xi_{k}
\,,\quad\mbox{with}\quad d\xi=\prod_{k=1}^{\nu L^{2}}d\xi_{k} \,.\label{Prob1}%
\end{equation}
The probability space is $(\Omega_{\nu}(L) , A_{\nu},\mu_{\nu})$, where $A_{\nu}$
is the $\sigma$-algebra of Borel sets of $\Omega_{\nu}(L)$. 

For simplicity of notation, we introduce the parameter
\begin{equation}
\theta=L\,\max\left\{  \frac{1}{\sqrt{B}},R\right\}  \label{theta}%
\end{equation}
that relates the three (non-dimensionalised) lengths. As mentioned in the introduction, 
$\theta$ is relevant in the cases where $B\gg 1$. In the other cases we use to simplify notation; 
it will indicate that the solutions depend on $B$, $R$ and $L$. 

For every $\omega\in\Omega_\nu(L)$ we denote by $\Gamma(\omega,\lambda,\theta)$ the interface 
solution of (\ref{S1E1}) and the Young condition. We recall that such an interface connects 
the domain $x_{1}=0$ to the domain
$x_{1}=L$ and that $\Gamma(\omega,\lambda,\theta)$ is the union of
non-intersecting elemental components that are solutions of (\ref{S1E1}) 
and the Young condition between a pair of grains.

In particular, each component of $\Gamma(\omega,\lambda,\theta)$ solves (\ref{A1E1}) 
with a particular value of $C$, and
we denote by $C(\Gamma)$ the maximum $C$ for each such $\Gamma$.

We index the elemental components of any $\Gamma(\omega,\lambda,\theta)$ following its orientation and 
let $J=\{1,\dots,m+1\}$ be this set of indexes if the interface connects $m$ grains. We shall write that a center 
$\xi_i\in \Gamma(\omega,\lambda,\theta)$ with $i=1,\dots,m$ if it is connected (by the Young conditions) 
to the elemental components $\gamma_{i-1}$ and $\gamma_{i}$ of $\Gamma(\omega,\lambda,\theta)$. 
We also denote by ${\cal D}_{i}=\|\xi_{i}-\xi_{i+1}\|$.  

We shall need a number of results. The first one is a consequence of the lemmas 
\ref{single:distance}, \ref{single:distance2} and \ref{single:distance3}, 
and will be used in sections \ref{regime:1} and \ref{regime:2}:

\begin{cor}
\label{corol1} For every $L$, $\omega\in\Omega_\nu(L)$ and $v_{0}$, there exists a
constant $D_{0}$, such that if a $\Gamma(\omega,\lambda,\theta)$ satisfies that
\[
\Gamma(\omega,\lambda,\theta)\cap\left\{  |\,x_{2}-v_{1}L|\,\geq\frac{D}{\sqrt{B}}\right\}
\neq\varnothing\text{ for }D\geq D_{0},
\]
then there exists a constant $K>0$ such that every pair of centers of grains
joined by $\Gamma$, $\xi_{i}$ and $\xi_{l}$, and contained in $\left\{|\,x_{2}-v_{1}L|\,\geq\frac{D}{\sqrt{B}}\right\}$, 
satisfy $\left\|  \xi_{i}-\xi_{l}\right\|  \leq\frac{K}{D\sqrt{B}}+2R$. In fact, $D_0\propto \sqrt{C(\Gamma)}$.   
\end{cor}

In many of the probability results we will use Stirling estimates (see e.g. \cite{FellerI}). Namely,
\begin{equation}\label{fact:stirling}
n! = \sqrt{2\pi n} \left(  \frac{n}{e}\right)  ^{n} e^{r_{n}} \quad\mbox{with}
\quad\frac{1}{12(n+1)}<r_{n}<\frac{1}{12n}\,.
\end{equation}
Thus we can write for any $n$, $m\in\mathbb{N}$ with $1<m<n$:
\begin{equation}\label{bino:stirling}
\binom{n}{m}=\frac{e^{r_{n}-r_{m}-r_{n-m}}}{\sqrt{2\pi}}\left(  
\frac{n}{(n-m)m}\right)  ^{\frac{1}{2}} (n)^{n}(m)^{-m}(n-m)^{-n+m}\,. 
\end{equation}
We have the following lemma:
\begin{lem}\label{comb:stirling}
Given $n_0$, $n$, $m\in \mathbb{N}$, with $n_0\leq m\leq n-1$ then
\[
\binom{n}{m}\leq C(n) e^{W(n,m)} \quad \mbox{with} \quad x = \frac{m}{n}\,, \quad
W(n,m) = n \Psi\left(\frac{m}{n}\right)-\frac{1}{2}\varphi\left(\frac{m}{n}\right)-\frac{1}{2}\log(n)
\] 
with
\[
\Psi(x)= -x\log(x) -(1-x)\log(1-x)  \,,\quad \varphi(x) = \log(x)+\log(1-x)\,, 
\]
and $C(n)\leq e^{-c/n}/\sqrt{2\pi}$ for some $c>0$. Moreover, at $m=n/2$, $W$ reaches its maximum 
$W(n,n/2)= -(n+1)\log(1/2)-(1/2)\log(n)>0$, attaining its minimum at either $m=n-1$ or $m=n_0$ 
(depending on the value of $n_0$ and $n$) for $n$ large enough. 
\end{lem}
\begin{proof}
The estimates follow from (\ref{bino:stirling}). First, let $K(n,m)=\exp(r_{n}-r_{m}-r_{n-m})\leq C(n)$ with 
\[
C(n)=\exp(\frac{ 1 }{ 12n }-\frac{ 1 }{ 12m+1 }-\frac{ 1 }{ 12(n-m)+1 }) \leq \exp( \frac{1}{12n} - \frac{1}{12n+1} -1 ) 
< \exp(-\frac{11}{12n})
\] 
where we use the estimates on $r_n$ (\ref{fact:stirling}). The estimate on the other factor of (\ref{bino:stirling}) 
follows by a calculus exercise: It is easy to check that $W(x)=n  \Psi(x)-\frac{1}{2}\varphi(x)$ has a maximum at $x=1/2$ 
if $n>3$ with $W(1/2)>0$. Also since $\lim_{x\to 0}W(x)=$ and $\lim_{x=1}=$ there must be at least two relative minima, 
one in $(0,1/2)$ and one in $(1/2,1)$. It is easy then to check that $W''(x)>0$ in $(0,1/2(1-\sqrt{(n-1)/(n+1)})$ and 
$W''(x)<0$ in $(1/2(1-\sqrt{(n-1)/(n+1)},1)$, 
where $x_l(n)=1/2(1-\sqrt{(n-1)/(n+1)}$ and $x_l(n)1/2(1+\sqrt{(n-1)/(n+1)}$ 
are the inflexion points. This means that there are two minima one is attained in $(0,x_l(n))$ 
and the other in $(x_r(n),1)$. 
Since $x_r(n)>(n-1)/n$ for all $n\geq 1$ and $x_l(n)<n_0/n$ with $1 \geq n_0\leq n$, the statement follows.
\end{proof}


\subsection{The regime $1\lesssim L\ll \min\{\sqrt{B},R^{-1}\}$}\label{regime:1}
The subcases of this regime can be treated similarly. 
Namely, the following theorem holds.

\begin{thm}
\label{theo:case1} For all $\varepsilon>0$ there exists $\theta_{\varepsilon}$
such that for all $\theta\leq\theta_{\varepsilon}$ there exists $\Omega_{\varepsilon}\subset\Omega_{\nu}(L)$ 
with $\mu_{\nu}(\Omega_{\varepsilon})\geq 1-\varepsilon$ and such that for any $\omega\in\Omega_{\varepsilon}$, the
solutions $\Gamma(\omega,\lambda,\theta)$ satisfy 
$\Gamma(\omega,\lambda,\theta)\equiv\left\{  x_{2}-v_{1}L=h(\omega)\right\}  $ with \[
\left|h(\omega)\right|  \leq \max\left\{  \frac{2D_{0}}{\sqrt{B}},2R\right\}  +2Kv_{0} L R^{2}\,,
\] 
where the positive constant $K$ is as in Corollary~\ref{corol1}.
\end{thm}

\begin{proof}
For simplicity of notation we introduce the parameter
\[
\delta(\theta)=\max\left\{  \frac{2D_{0}}{\sqrt{B}},2R\right\}  +2Kv_{1}LR^{2}\,,
\]
that satisfies $\delta(\theta)\to0$ as $\theta\to 0$.

We first compute the probability of having a configuration $\omega$ without 
centers in the strip $|x_{2}-v_{0}L|<\delta$. We call the set of such
configurations $\Omega_{0}$, then, clearly,%
\[
\mu_{\nu}(\Omega_{0})=\left(  1-\frac{2\delta}{L}\right)^{\nu L^2}\to1\quad\mbox{as}
\quad\theta\to0\,.
\]
On the other hand if there exist $\omega\in\Omega_{0}$ and a $\Gamma
(\omega,\lambda,\theta)$ such that all its points satisfy $\left|  x_{2}%
-v_{0}L\right|  \geq\delta$ then, by Corollary~\ref{corol1}, every two
consecutive grains, $\xi_{i}$ and $\xi_{l}$, joined by $\Gamma$, satisfy
$\left\|  \xi_{i}-\xi_{l}\right\|  \leq M(\theta)$ where $M(\theta)=\frac
{K}{D_{0}\sqrt{B}}+2R$ (observe that $M(\theta)\to0$ as $\theta\to0$). The
number of grains contained in $\Gamma$ is at least $N_{0}=\frac{L}{M(\theta)}$. 
This implies that the probability of such an event tends to $0$ as $\theta\to 0$.
\end{proof}

\subsection{The regime $L=O(\min\{\sqrt{B},R^{-1}\})$ and $L\to\infty$}\label{regime:2}
We first observe that if we denote by $U_{0}$ the set of
configurations for which the 'horizontal´ solutions exist (no grains, therefore,
are connected), then $\mu_{\nu}(U_{0})=1-\mu_{\nu}(\omega_{0})$ where
$\omega_{0}$ is the set of configurations with at least one grain in the strip
around the horizontal reference height of width $2R$, thus $\mu_{\nu}(\omega
_{0})=2R/L$ and $\mu_{\nu}(U_{0})=1-2R/L$. In the current limit this means
that $\mu_{\nu}(U_{0})\to 1$ as $L\to\infty$.

If we know let $U_{1}$ be the set of configurations that have at least a
solution $\Gamma$ with exactly one grain connected to the walls, 
then $\mu_{\nu}(U_{1})\leq\mu_{\nu}(\omega_{1})$ where $\omega_{1}$ is the set of 
configurations that have at
least one center at a distance smaller than or equal than $2/\sqrt{B} + R$ 
to the horizontal reference line (recall
that the supremum of the distance to the horizontal that can be reached is
precisely $2/\sqrt{B}$ for the solution with $C=1$ different from the
horizontal line). Then $\mu_{\nu}(\omega_{1})=2(2/\sqrt{B} + R)/L\to0$ as
$L\to\infty$ and thus also $\mu_{\nu}(U_{1})\to0$ as $L\to\infty$.

In the next theorem we obtain that in general $\mu_{\nu}(U_{k})\to 0$ as $L\to \infty$ 
for all $k>0$ and that an upper bound for such behaviour is given by a Poisson distribution. 

First let us make precise the definition of the sets of configurations $U_{k}$:
\begin{defin}
Let $U_{k}\subset\Omega_{\nu}(L)$ be the set of configurations $\omega$ for which there is at 
least one solution $\Gamma(\omega,\lambda,\theta)$ such that 
$k=\#\left\{\xi_{j}\in\omega:\ \Gamma(\omega,\lambda,\theta)\cap B_{R}(\xi_{j})\neq\emptyset \right\}$.
\end{defin}

\begin{thm}\label{theo:case2} 
If $k\geq 1$, there exists $N_0\leq2D_{0}\theta\nu <\nu L^{2}$ and $L$ sufficiently large 
such that 
\[
 \mu_{\nu}(U_{k})\lesssim   e^{-2D_{0}\theta \nu} \sum_{i=k}^{N_0}\frac{1}{i!} (2D_{0}\theta \nu)^{i} \,.
\]
Moreover, for all $\varepsilon>0$, there exists a $L_{\varepsilon}>0$ such that for all $L\geq L_{\varepsilon}$ 
there exists $\Omega_{\varepsilon}\subset\Omega_\nu(L)$ such that $\mu_{\nu}(\Omega_{\varepsilon})\geq 1-\varepsilon$, 
and that for all $\omega\in\Omega_{\varepsilon}$ the solutions $\Gamma(\omega,\lambda,\theta)$ satisfy 
$|C(\Gamma)-1|\leq\varepsilon$.
\end{thm}

\begin{proof}
First, we observe that the probability of having at least $k$ connections is the sum of 
the probabilities of having exactly $i$ connections with $i=k\dots \nu L^2$. 
Then we distinguish between the configurations that have connections 
outside the strip around the horizontal reference line of width $2D_0R$ 
and those that do not. We first show that the probability of having 
connections outside the given strip tends to $0$. On the other hand,
 the probability of having an interface connecting exactly $k$ grains that lies
 inside a strip of width $2D_0R$ is less than that of having $k$ grains in the strip of that same width. 
Then we compute this probability. We recall the notation $N=\nu L^2$ that we use throughout this proof.

First, we treat the case $R\gg1/\sqrt{B}$. Let $\omega_s\subset\Omega_\nu(L)$ be the set of configurations 
that have a solution $\Gamma(\omega,\lambda,\theta)$, with $\omega\in\omega_s$, that intersects grains 
$\xi$ with $|\xi^{(2)}-v_{1}L| >D_{0}R$. In that case, there are at least three centers $\xi_{i}$, $\xi_{j}$, 
$\xi_{l}\in\omega\subset \omega_s$ with $\| \xi_{i}-\xi_{j}\|<(2+\tilde{K})R$ and, by Corollary~\ref{corol1}, 
$\| \xi_{j}-\xi_{l}\| <(2+\tilde{K})R$ 
for some $\tilde{K}$ of order $1$. The probability of such events is:%
\begin{align*}
&\mu_{\nu}(\omega_s)\leq\binom{N}{3}\frac{(L^{2})^{N-2}}{(L^{2})^{N}}
\left(  \int_{\|\xi_{1}-\xi_{2}\| <(2+\tilde{K})R}d\xi_{2}\right)  
\left(  \int_{\| \xi_{2}-\xi_{3}\| <(2+\tilde{K})R}d\xi_{3}\right)
\\
&  = \frac{\nu}{6} (N - 1)(N - 2)\pi^{2}(2+\tilde{K})^{4}R^{4}
  \sim \bar{K} \frac{1}{L^{2}}
\quad\mbox{as}\quad L\to\infty\,.
\end{align*}
Let $\omega_s^c$ be the complementary event of $\omega_s$. Then, $ \mu_\nu(\omega_s^c)\to 1$ as $L\to\infty$, 
and this together with Remark~\ref{C1:complete} (ii) shows the second statement of the Theorem.

We let $\omega_{k}$ the set of configurations that have exactly $k$ grains in a strip of width $2D_{0}R$. 
Then, the probability measure of $\omega_{k}$ is
\begin{align*}
& \mu_\nu(\omega_{k}) =\frac{1}{L^{2N}}
\sum\limits_{\substack{\left\{  l_{1},...,l_{k}\right\}\subset\{1,..,N\}\\l_{i}\neq l_{j}}}
\left( 
\prod\limits_{i=1}^{k}\int_{|\xi_{l_{i}}^{(2)}-v_{1}L | \leq D_{0}R}d\xi_{l_{i}}\right)  
\left( 
 \prod\limits_{\substack{m=1\\m\neq l_{i}}}^{N}\int_{ |\xi_{m}^{(2)}-v_{1}L | >D_{0}R}d\xi_{m}
\right)
\\
&  =\frac{1}{L^{2N}}\binom{N}{k}\left( \prod\limits_{i=1}^{k}
\int_{|\xi_{i}^{(2)}-v_{1}L| \leq D_{0}R}d\xi_{i}\right)  
\left(\prod\limits_{m=k+1}^{N}
\int_{| \xi_{m}^{(2)}-v_{1}L|>D_{0}R}d\xi_{m}\right) \\
&  =\frac{1}{L^{N}}\binom{N}{k}\left(  \prod\limits_{i=i}^{k}\int_{|\xi_{i}^{(2)}-v_{1}L| \leq D_{0}R}d\xi_{i}^{(2)}\right)  
\left(\prod\limits_{m=k+1}^{N}\int_{|\xi_{m}^{(2)}-v_{1}L|>D_{0}R}d\xi_{m}^{(2)}\right) \\
&  =\frac{1}{L^{N}}\binom{N}{k}(2D_{0}R)^{k}(L-2D_{0}R)^{N-k}=\binom{N}{k}
\left(  \frac{2D_{0}R}{L}\right)  ^{k}\left(  1-\frac{2D_{0}R}{L}\right)^{N-k}\\
&  =\binom{\nu L^{2}}{k}\left(  \frac{2D_{0}\theta}{L^{2}}\right)^{k}
\left(1-\frac{2D_{0}\theta}{L^{2}}\right)  ^{\nu L^{2}-k}\,.
\end{align*}
Where we have used that, in the current case, $\theta= R L$.

In a strip of width $2D_{0}R$ and the maximum number of
centers expected is proportional to $2\nu L D_{0}R =2\nu D_{0}\theta=:N_0=O(1)$. 
To see this, we let $U_{P\geq  N_0}$ be the set of configurations for which the number
 of particles in the strip of width $2D_0 R$ is larger or equal to $N_0$. Then
\[
\mu_\nu(U_{P\geq N_0})= \sum_{j=N_0}^{\nu L^2} \binom{\nu L^2}{j} 
\left(\frac{2D_0\theta}{L^2}\right)^j \left(1-\frac{2D_0\theta}{L^2}\right)^{\nu L^2-j} 
\]
Using Lemma~\ref{comb:stirling} we can estimate each term (except the last one which is 
$(2D_0\theta/L^2)^{\nu L^2}$) as follows:
\[
\binom{\nu L^2}{j} \left(\frac{2D_0\theta}{L^2}\right)^j \left(1-\frac{2D_0\theta}{L^2}\right)^{\nu L^2-j} 
\leq C(\nu L^2) \exp\left( W\left(\nu L^2,\frac{j}{\nu L^2}\right) \right)
\]
where $C(\nu L^2)$ is as in Lemma~\ref{comb:stirling}, and
\begin{align*}
W(\nu L^2, x) =& -\left(\nu L^2 +\frac{1}{2}\right)\big(x\log(x)+(1-x)\log(1-x)\big) -\frac{1}{2}\log(\nu L^2)\\
& + \nu L^2\big( x\log(2D_0\theta/L^2 ) + (1-x)\log(1-2D_0\theta/L^2)\big)\quad 
\mbox{with} \quad x\in \left[\frac{2 D_0\theta}{L^2},1\right)\,.
\end{align*}
Then, $W(\nu L^2,x)\to -\infty$ as $L\to \infty$, in particular, $W(\nu L^2,x) \sim \nu L^2 x\log(2D_0\theta/L^2 )$ 
as $L\to \infty$ if $x=O(1)$ and $W(\nu L^2,x)\sim  -\frac{1}{2}\log(\nu L^2)$ as $L\to \infty$ 
if $x\sim \frac{2D_0\theta}{L^2}$. Summing up all the terms implies that $\mu_\nu(U_{P\geq N_0})\to 0$ as $L\to \infty$ 
at an exponential rate.

Finally, we write 
\[
U_k= (U_k \cap \omega_s)\cup  (U_k \cap \omega_s^c)=(U_k \cap \omega_s)\cup  (U_k \cap \omega_s^c\cap U_{P\geq N_0} )\cup  (U_k \cap \omega_s^c\cap U_{P<N_0})
\]
where $U_{P <N_0}$ denotes 
the set of configurations for which the number of particles in the strip is less than $N_0$. Then,
\[
\mu_\nu(U_{k})\leq \mu_\nu(\omega_s)+\mu_\nu(U_{k}\cap \omega_s^c\cap U_{P\geq N_0})
+\sum_{i=k}^{N_0}\binom{\nu L^{2}}{i}
\left(\frac{2D_{0}\theta}{L^{2}}\right)^{i}\left(1-\frac{2D_{0}\theta}{L^{2}}\right)^{\nu L^{2}-i}\,.
\]
As we have seem the first two terms tend to $0$ as $L\to\infty$, and the third tends to 
$e^{-2D_{0}\theta \nu}\sum_{i=k}^{N_0}\frac{1}{i!}\left(2D_{0}\theta \nu\right)^{i}$ as $L\to\infty$.

The case $R\ll1/\sqrt{B}$ is analogous but working on a strip of width $2D_{0}/\sqrt{B}$.

\end{proof}


\subsection{The regime $L \gg \sqrt{B}$ with $B\to \infty$ and $R\to 0$} 
\label{regime:3}
In this regime we shall prove that the probability of having configurations with solutions that
separate from the horizontal reference line a distance of order one tends to zero in a regime where $L$ 
is not too large compare to some increasing function of $B$.


We prove the result in the case $1/\sqrt{B}\ll R$, the other cases can be shown similarly, as we remark 
at the end of this section.

Let us give two auxiliary lemmas, that will be needed in the proof of the main result of this section. 
The first one is a technical lemma:
\begin{lem}
\label{sqrt:growth}
Let $b>0$ and a sequence $\{a_k\}$ such that 
\[
a_k + \frac{1}{2} \frac{b}{1+a_k}< a_{k+1} < a_k +  2\frac{b}{1+a_k}\,,
\]
and $a_0 > \delta>0$. Then, there exist positive constants $C_1$ and $C_2$ independent of $k$, such that
\[
C_2 \sqrt{k} < a_k < C_1 \sqrt{k} \,.
\]
\end{lem}
The proof follows by induction, multiplying the inequalities by $a_k$ and $a_{k+1}$ and summing up the results, the constants 
$C_1$ and $C_2$ satisfy $C_1> 4 b ( 1+ b / C_2)$ and $C_2 >  \max_x ( b \sqrt{x} + \sqrt{b^2 x + b} ) / 2( 1 + C_1 \sqrt{x} ) $.

The next lemma is a geometric result, which we prove in Appendix~\ref{sec:proofgeolemma}.
\begin{lem}\label{geometric:lemma} 
If $\Gamma(\omega, \lambda,\theta)$ is a solution with $m$ components joining ($m-1$) 
grains, $\xi_1,\dots,\xi_{m-1}$, ordered following the orientation of $\Gamma(\omega, \lambda,\theta)$, then
\begin{equation}
\sum_{i=0}^{m}\| \xi_{i}-\xi_{i+1}\| < 4m \left(2R +\frac{K_{1}}{\sqrt{B}}\right) + K_{2} L \,, \label{geometric}%
\end{equation}
for some positive constants $K_{1}$ and $K_{2}$ independent of $L$, $B$ and $R$.
\end{lem}




Next, we define the configurations that may allow solution interfaces that separate a distance of order one from 
the reference horizontal line.
 The centers of such configurations that might be joined by a solution interface must satisfy certain conditions, 
according to Corollary~\ref{corol1}. We gather these in the following definition.

\begin{defin}
\label{bcs} We let the union of the horizontal reference line to the left boundary and to right boundary be denoted by
\[
Z_{L}^{l}   =\left\{  x_{2}=v_{1}L\right\}  \cup\left\{  x_{1}=0\right\}
\,,\quad 
Z_{L}^{r}   =\left\{  x_{2}=v_{1}L\right\}  \cup\left\{  x_{1}=L\right\}  \,.
\]

\end{defin}

\begin{defin}\label{centers:gamma} 
For given constants $K>0$ and $h>0$ of order one, and a given $\theta$ in the current regime, 
we let $U_h(K,D_{0},\theta,L)\subset\Omega_\nu(L)$ be the set of configurations of particles 
such that for every $\omega\in U(K,D_{0},\theta,L)$ there exists $N_0\leq \nu L^2$ and an injective map 
$g:\{1,\dots,N_0\}\rightarrow\{1,\dots,\nu L^2\}$ so that the centers $\{\xi_{g(n)}\}_{n=1...N_0}$ 
satisfy the following conditions:
\begin{enumerate}
\item $\max_{n}\{|\xi_{g(n)}^{(2)}-v_{1}L|\}\geq h$. 

\item 
\[
| \xi_{g(n+1)}^{(2)}-\xi_{g(n)}^{(2)}| \leq
\frac{1}{\sqrt{B}}\frac{K}{(1+\sqrt{B}|\xi_{g(n)}^{(2)}-v_{1}L|)}\,.
\]

\item If $\max\{|\xi_{g(n)}^{(2)}-v_{1}L|,|\xi_{g(n+1)}^{(2)}%
-v_{1}L|\}\geq D_{0}/\sqrt{B}$ then%
\[
\| \xi_{g(n+1)}-\xi_{g(n)}\| \leq\frac{1}{B}\frac{K}{|\xi_{g(n)}^{(2)}-v_{1}L|}\,.
\]

\item 
\begin{align*}
dist(\xi_{g(1)},Z_{L}^{l})  &  
\leq D_{0}\min\left\{  \frac{1}{\sqrt{B}},\frac{1}{B|\xi_{g(1)}^{(2)}-v_{1}L|}\right\}  \,,\\
dist(\xi_{g(N_0)},Z_{L}^{r})  &  
\leq D_{0}\min\left\{  \frac{1}{\sqrt{B}},\frac{1}{B|\xi_{g(N_0)}^{(2)}-v_{1}L|}\right\}  \,.
\end{align*}

\item Moreover, they satisfy (\ref{geometric}).
\end{enumerate}
\end{defin}

We observe that condition {\it (ii)} is suggested by the statement of Corollary~\ref{corol1} 
and the values of $x_2^{max}$, $x_2^{med}$ in (\ref{x2:mim}) and (\ref{x2:med:min}). 
Condition {\it (iii)} corresponds is also suggested by Corollary~\ref{corol1}. Finally, 
condition {\it (iv)} implies that the sequence starts either near the horizontal or near $Z_{L}^{l}$ 
and ends either near the horizontal or near $Z_{L}^{r}$.

We can now prove the following theorem:
\begin{thm}\label{theo:case3}  
Let $\Omega_{h}\subset\Omega_\nu(L)$ denote the set of all configurations $\omega$ such that there exists 
a $\Gamma(\omega,\lambda,\theta)$ satisfying $\max|\xi_{j}^{(2)}-v_{1}L|\,\geq h$ for some $h>0$ of order one.
Then, if $\sqrt{B}\ll L\ll\sqrt{B}\log B$ as $B\to\infty$, $\mu_\nu(\Omega_{h})\to 0$ as $B\to\infty$. 
\end{thm}

\begin{proof}
Definition~\ref{centers:gamma} implies that $\mu_\nu(\Omega_{h})\leq \mu_\nu( U_h(K,D_{0},\theta,L)  )$. 
Let us then get an estimate on $\mu_\nu( U_h(K,D_{0},\theta,L)  )$. 

We denote by $\Sigma_{m}:=\left\{g:\{1,\dots,m\}\rightarrow\{1, \dots,N\}: \ \text{injective}\right\}$ 
and for every $m\leq N$ and every $g\in\Sigma_{m}$ we let $U^{g}\subset
U_h(K,D_{0},\theta,L)$ be the set configurations for which there exists a family of centers 
$\{ \xi_{g(n)}\}_{n=1,\dots,m}$ satisfying the conditions of Definition~\ref{centers:gamma}. 
Observe that these sets of configurations are not, in general, disjoint. Then
\[
U_h(K,D_{0},\theta,L)=\bigcup\limits_{m=1}^{N}\bigcup\limits_{g\in\Sigma_{m}}U^{g}\,,
\]
and we estimate its probability as follows
\begin{align*}
&\mu_\nu\big(U_h(K,D_{0},\theta,L)\big)=
\frac{1}{L^{2N}}\left(\prod\limits_{n=1}^{N}\int_{U_h(K,D_{0},\theta,L)}d\xi_{n}\right)\\ 
&\leq \frac{1}{L^{2N}}\sum\limits_{m=1}^{N}\sum\limits_{g\in\Sigma_{m}}
\left(\prod\limits_{n=1}^{N}\int_{U^{g}}d\xi_{n} \right)=
\frac{1}{L^{2N}}\sum\limits_{m=1}^{N}N(N-1)\dots(N-m+1)\left( \prod\limits_{n=1}^{N}\int_{U^{\bar{g}}}d\xi_{n} 
\right)
\\
& \leq\sum\limits_{m=1}^{N}\frac{N(N-1)\dots(N-m+1)}{L^{2m}}
\left(\prod\limits_{n=1}^{m}\int_{U^{\bar{g}}}d\xi_{n}\right)\,,
\end{align*}
where we have reordered the first indexes by setting $\bar{g}(j)=j$ with $j\in\{1,\dots,m\}$ 
for every $m$, and in the last inequality we have simply counted the number of elements in each $\Sigma_{m}$ and 
also estimated the measure of the elements $\xi_{n}$ with $n>m$ by the total
measure $L^{2}$.

Now, we consider two types of configurations in each $U^{\bar{g}}$, the
ones for which all points $\xi_{n}$ satisfy that%
\[
\left|  \xi_{n}^{(2)}-v_{1}L\right|  \geq\frac{D_{0}}{\sqrt{B}}\quad
\mbox{for all} \quad n\in\{1,\dots,m\}\,,
\]
and the ones for which there exists at least one $j\ $with $\left|  \xi
_{j}^{(2)}-v_{1}L\right|  <\frac{D_{0}}{\sqrt{B}}$. In the first case, we
have, due to the properties in Definition~\ref{centers:gamma}, that
\[
\left\|  \xi_{n}-\xi_{n+1}\right\|  \leq\frac{K}{D_{0}\sqrt{B}}\,,\quad
\mbox{for}\quad n\in\{0,\dots,m\}\,,
\]
where we abuse notation by referring to the vertical lines $x_{1}=0$ and
$x_{1}=L$ by $\xi_{0}$ and $\xi_{m+1}$, respectively, and denoting the distance
of points to these lines with the norm symbol. Observe that then $m\geq
LD_{0}\sqrt{B}/K$.

Let us denote by 
\[U_{1}^{\bar{g}} = U^{\bar{g}}\cap\{\omega\in\Omega_\nu(L): \, |\xi_{n}^{(2)}-v_{1}L |  
\geq\frac{D_{0}}{\sqrt{B}}\;\forall n\in\{1,\dots,m\} \}\] 
and 
\[U_{2}^{\bar{g}}=U^{\bar{g}}\setminus\{\omega\in\Omega_\nu(L):\, |\xi_{n}^{(2)}-v_{1}L|
\geq\frac{D_{0}}{\sqrt{B}}\;\forall n\in\{1,\dots,m\} \}\,,
\] 
then
\begin{equation}
\label{Ies}
\mu_\nu\big(U_h(K,D_{0},\theta,L)\big)    \leq    I_{1} +I_{2}
\end{equation}
with
\begin{equation}\label{Ies2}
I_j = \sum\limits_{m=M_j}^{N} \frac{N(N-1) \dots(N-m+1)}{L^{2m}}
 \prod\limits_{n=1}^{M}\int_{U_{j}^{\bar{g}}}  d \xi_{n}\,, \quad j=1,2\,,
\end{equation}
where $M_1\geq LD_{0}\sqrt{B}/K$ and $M_2$ will be determined bellow.

We first estimate $I_1$. For $B$ large enough so that $\nu \pi K^{2}/((D_{0})^{2}B)<1$,
we have that
\begin{align*}
I_{1}  &  \leq\sum\limits_{m\geq LD_{0}\sqrt{B}/K  }^{N}N(N-1)\dots(N-m+1) \frac{L^2}{L^{2m}}
\,\pi^{m-1}\left(  \frac{K}{D_{0}\sqrt{B}}\right)  ^{2m-2}\\
&  \leq \nu L^{2}\sum\limits_{m\geq LD_{0}\sqrt{B}/K}^{N} 
\left(  \frac{\nu\pi K^{2}}{(D_{0})^{2}B}\right)  ^{m-1}\leq 
L^{2} \left(  \frac{\nu\pi K^{2}}{(D_{0})^{2}B}\right)  ^{\frac{LD_{0}\sqrt{B}}{K}}\,,
\end{align*}
and this tends to zero provided that
\begin{equation}
(B)^{L\sqrt{B}}\gg L^{2}\,. \label{theta:cond1}%
\end{equation}
In the current limit, $L^{2}\gg B$ and $B\rightarrow\infty$, then for
large enough $\theta$ it is clear that $(B)^{L\sqrt{B}}>(B)^{L}$ and so 
(\ref{theta:cond1}) holds and $I_1 \to 0$ as $L\to \infty$.

In order to estimate $I_{2}$, for a $\omega\in U_{2}^{\bar{g}}$ we define
a subset of the points in $\omega$ as follows. First observe that, by the
definition of $U_{2}^{\bar{g}}$, there is a $j_{0}\in\{1,\dots,m\}$, such
that $\| \xi_{j_{0}}^{(2)}-v_{1}L\| <\frac{D_{0}}{\sqrt{B}}$, on the other 
hand there exist a $j_{1}\in\{1,\dots,m\}$ that 
$\| \xi_{j_{1}}^{(2)}-v_{1}L \|>2D_{0}(R+1/\sqrt{B})$. Then we can define
the first point of the subset by taking $\xi_{l_{1}}=\xi_{j_{0}}$ and
$\xi_{l_{T}}=\xi_{j_{1}}$ (the last point of the subset). For
each $k\in\{1,\dots,M_2\}$, we choose the points as follows%
\[
l_{k+1}=\min\left\{  j>l_{k}:\;|\xi_{l_{k+1}}^{(2)}-v_{1}L|\,>
|\xi_{l_{k}}^{(2)}-v_{1}L|\,+\frac{K}{2\sqrt{B}}\frac{1}{(1+\sqrt{B}|\xi_{l_{k}}^{(2)}-v_{1}L|)}\right\}\,,
\]
then by the properties of Definition~\ref{centers:gamma}, and this definition,
the subset satisfies:%
\begin{align*}
\;|\xi_{l_{k+1}}^{(2)}-v_{1}L|\,  &  >|\xi_{l_{k}}^{(2)}-v_{1}L|\,
+\frac{K}{2\sqrt{B}}\frac{1}{(1+\sqrt{B}|\xi_{l_{k}}^{(2)}-v_{1}L|)}\\
\;|\xi_{l_{k+1}}^{(2)}-v_{1}L|\,  &  <|\xi_{l_{k}}^{(2)}-v_{1}L|\,
+2\frac{K}{\sqrt{B}}\frac{1}{(1+\sqrt{B}|\xi_{l_{k}}^{(2)}-v_{1}L|)}\,,
\end{align*}
and also there exist a $k_{0}$ such that $|\xi_{l_{k_{0}}}^{(2)}-v_{1}L|>D_{0}/(2\sqrt{B})$.

We can apply Lemma~\ref{sqrt:growth} with $a_k= \sqrt{B}|\xi_{l_{k}}^{(2)}-v_{1}L|$ 
to obtain an estimate on $M_2$, namely,
\[
h <C_1\frac{\sqrt{T}}{\sqrt{B}}\,,
\]
so
\[
M_2>M_{0}:=\left(\frac{h}{C_1}\right)^{2} B
\,.
\]

We then estimate the second term in (\ref{Ies}) as follows,
\begin{align*}
I_{2}  &  \leq\sum\limits_{m\geq M_{0}}^{N}\frac{N(N-1)\dots(N-m+1)}{L^{2m}}
\left( \prod\limits_{n=1}^{M_{0}}  \int_{U_{2}^{\bar{g}}\cap\{\xi_{l_{k}}\}_{k}}   d\xi_{n} \right)  
\left( \prod\limits_{n=1}^{m-M_{0}} \int_{U_{2}^{\bar{g}}\smallsetminus\{\xi_{l_{k}}\}_{k}}  d\xi_{n}\right) \\
&  \leq\sum\limits_{m\geq M_{0}}^{N}\frac{N(N-1)\dots(N-m+1)}{L^{2m}}
\left(\prod\limits_{n=1}^{M_{0}} \int_{U_{2}^{\bar{g}}\cap\{\xi_{l_{k}}\}_{k}} d\xi_{l_{n}}^{(2)}     \right)  
\left(\prod\limits_{n=1}^{m-_M{0}}  \int_{U_{2}^{\bar{g}}\smallsetminus\{\xi_{l_{k}}\}_{k}} d\xi_{m_{n}}^{(2)}\right)  
\left(\prod\limits_{n=1}^{m}  \int_{U_{2}^{\bar{g}}}  d\xi_{n}^{(1)}   \right) \\
&  \leq \nu L^3 \sum\limits_{m\geq M_{0}}^{N} \nu^{m-1}
\left(  \prod\limits_{n=2}^{M_{0}}\frac{2K}{C_{2}\sqrt{B}\sqrt{n}}\right)
\left(  \prod\limits_{n=2}^{m}{\cal D}_{n}^{2}\right)= 
\nu L^{3}\frac{1}{\sqrt{M_{0}!}}\left(\frac{2K}{C_{2}\sqrt{B}}\right)^{M_{0}-1}\sum\limits_{m\geq M_{0}}^{N}
\left(  \prod\limits_{n=2}^{m} \sqrt{\nu}{\cal D}_{n}\right)  ^{2}\,.
\end{align*}
Here we have indexed the elements of $U_{2}^{\bar{g}}\smallsetminus\{\xi_{l_{k}}\}_{k}$ by $m_n$ with 
$n\in\{1,\dots,m-M_0\}$. Also, in the last inequality, we have used that the elements of $\{\xi_{l_{k}}\}_k$ 
satisfy $|\xi_{l_{k+1}}^{(2)} -\xi_{l_k}^{(2)} |\leq   2K / B|\xi_{l_k}-v_1 L|$ and Lemma~\ref{sqrt:growth}.
We recall that ${\cal D}_{n}=\| \xi_n-\xi_{n+1}\|$.

Then, using the inequality of arithmetic and geometric means and Lemma~\ref{geometric:lemma}, 
we obtain
\begin{align*}
I_{2}  
&  \leq \nu L^{3}\frac{1}{\sqrt{M_{0}!}}\left(\frac{2K}{C_{2}\sqrt{B}}\right)^{M_{0}-1}\sum\limits_{m\geq M_{0}}^{N}
\left(  \frac{1}{m-1}\sum_{n=2}^{m} \sqrt{\nu} {\cal D}_{n}\right)  ^{2(m-1)}\\
&  \leq L^{3}\frac{1}{\sqrt{M_{0}!}}\left(  \frac{2K}{C_{2}\sqrt{B}}\right)^{M_{0}-1}\sum\limits_{m\geq M_{0}}^{N}
\left(  \frac{\nu}{m-1}\left(m\frac{\bar{K}_1}{\sqrt{B}} + K_2 L\right)  \right)  ^{2(m-1)}\,.
\end{align*}
We rewrite
\[
\left(  \frac{\nu}{m-1}\left(m\frac{\bar{K}_1}{\sqrt{B}} + K_2 L\right)  \right)  ^{2(m-1)} = 
\left(\frac{\nu \bar{K}_1}{\sqrt{B}} \right)^{2(m-1)}
 \left(\frac{m}{m-1} + \frac{K_2}{\bar{K}_1}\frac{L\sqrt{B}}{m-1} \right)^{2(m-1)}\,,
\]
then, for all $m\geq M_{0}$ and $B$ large enough
\[
I_{2}\leq L^{3}\frac{1}{\sqrt{M_{0}!}}\left(\frac{\nu \bar{K}_1}{\sqrt{B}} \right)^{2(M_0-1)} e^{2\left(\frac{K_2}{\bar{K}_1} L\sqrt{B} +1\right)}\,.
\]

We recall that $M_{0}\propto B$ and using (\ref{fact:stirling}) we can conclude
 that $I_2$ tends to zero as $L\to \infty$ provided that largest exponential 
growing factor is controlled by the fastest exponentially decreasing one, 
that is if $L\sqrt{B}\ll B\log B$ as $L\to \infty$, and this implies the result.

\end{proof}

Theorem~\ref{theo:case3} can be formulated in the case $1/\sqrt{B}\lesssim R$ by replacing $\sqrt{B}$ 
by $R^{-1}$ when $R\to 0$. The proof is analogous, except for the proof of Lemma~\ref{geometric:lemma}. 
In this case, the conclusion is simpler to achieve, by realising that the longest part of an interface 
joining the least number of grains is one crossing $x_2=v_1 L$, all other parts are away from this one 
a distance of order $1/\sqrt{B}$, applying Corollary~\ref{corol1} the lemma follows.


\subsection{The regime $L=O(1/B)$ as $B\to 0$}\label{regime:4}
We only consider here the case $L=O(1/B)$, and to simplify we assume that $L=a/B$ 
for some $a>0$ of order one. The reason for this is that we expect the maximum difference
in height in a full solution $\Gamma$ to be of order $1/B$. Indeed, if $\Gamma$ reaches
 distances from $x_2=\lambda$ much larger than $1/B$, the elemental components of $\Gamma$ 
on that region would have a very small radius of curvature, and could not join grains  
at an average distance of order one. 

In this section we further restrict the height of the possible $\Gamma$'s 
to the interval $[\lambda-\varepsilon_{0}L,\lambda+\varepsilon_{0}L]$ where, essentially 
\[
0 < \varepsilon_{0} <  \min\left\{ v_{1}, 1-v_{1},\frac{1}{2a}\right\}
\]
and (for reasons that become apparent below) $\varepsilon_0\in\mathbb{Q}$. 
This ensures that the curvature of $\Gamma$ in this region is small enough 
and will allow to approximate the elemental components of $\Gamma$ by 
straight lines, since for $x_2$ in this interval $|H| \leq a\varepsilon_0<1/2$.
We recall, however, that this restriction depends on the configuration of 
grains, since $\lambda=v_1 L$ depends on $\Gamma$. We shall see that we can fix 
a volume $V$ that is close to $v_1$ for the configurations of interest.

We show that, in this regime, given {\it any} regular curve without 
self-intersections contained in 
$[0,L]\times[\lambda - \varepsilon_{0}L,\lambda + \varepsilon_{0}L]$
there are configurations $\omega\in\Omega_\nu(L)$ 
with probability close to one, such that the interface solution 
$\Gamma(\omega,\lambda,\theta)$ is close to such curve 
in some suitable sense. 

Before we make precise this theorem, let us give the main idea of the proof and 
establish some settings and definitions. We make a partition of the domain into small 
squares $Q$. For any given smooth curve $\Lambda$ without self-intersections
 and compatible with the prescribed volume condition, we show that the probability 
that a solution interface $\Gamma$ satisfies $\Gamma\cap Q \neq \emptyset$ for every 
$Q$ that has $Q\cap \Lambda \neq \emptyset$ tends to one as $L\to \infty$. 
This means that all curves compatible with the fixed volume condition can be approximated 
by an interface $\Gamma$. The proof of that the probability of $\Gamma\cap Q\neq \emptyset$ 
if $Q\cap \Lambda\neq \emptyset$ tends to zero, will be formulated as a percolation problem, 
see below for the details. 

We can now state the main theorem of this section, but first we define the notion of compatible curves. 
\begin{defin}[Compatible curves]\label{comp:curve}
Let $V\in(0,1)$ and $\varepsilon_0\in(0,1)\cap \mathbb{Q}$ with 
\[
\varepsilon_0<\min\left\{V,1-V,\frac{1}{2a}\right\}\,.
\] 
We say that a $C^1$ curve without self-intersections 
$\Lambda:[0,1]\rightarrow [0,1]^2$ is compatible with $(V,\varepsilon_0)$ if 
$\Lambda([0,1])\subset [0,1]\times [V-\varepsilon_0,V+\varepsilon_0]$, 
$\Gamma(0)\in\{ 0\}\times[V-\varepsilon_0,V+\varepsilon_0]$, 
$\Gamma(1)\in\{ 1\}\times[V-\varepsilon_0,V+\varepsilon_0]$ and the area under
$\Gamma$ is equal to $V$.
\end{defin}

We remark that we compare real interface solutions $\Gamma$ to compatible 
curves in the rescaled domain $[0,1]\times[V-\varepsilon_0,V+\varepsilon_0]$.
We shall denote by $\tilde{\Gamma}$ the result of expressing $\Gamma$ in the variables $x'= x/L$. Then:

\begin{thm}\label{percolation:theo} 
Assume that $B=a/L$ for some $a>0$. Given $V\in(0,1)$ let $\varepsilon_0$ as in Definition~\ref{comp:curve} and let 
$\Lambda\in C^{1}([0,1])$ be a curve compatible with $(V,\varepsilon_0)$.
Then, there exist $\nu_0>0$ and $L_0>0$ such that for all $\nu\geq \nu_0$ and $L\geq L_{0}$ there exists 
$\alpha_{\max}(R)\geq 0$, such that $\alpha_{max}(R)\to\pi/ 2$ as $R\to 0$, and ${\cal U}\in \Omega_\nu(L)$ and a 
$\varepsilon(L)>0$ with $B\ll\varepsilon(L)$, $\varepsilon(L) \to 0$ as $L\to\infty$ such that
\[
\mu_{\nu}({\cal U}) \geq \delta_{\varepsilon}(L) \quad \mbox{with}  \quad 0<
\delta_{\varepsilon}(L)\longrightarrow 1^{-} \quad \mbox{as} \quad  L\to  \infty\,,
\]
and that for any $\omega\in {\cal U}_{L}$, there exists a solution interface 
$\Gamma(\omega,\lambda,\theta)$ with contact angle $\alpha\in[-\pi/2,\alpha_{max}(R)]$ that satisfies
\[
\tilde{\Gamma}(\omega,\lambda,\theta)\subseteq T_{\sqrt{2}\varepsilon(L)}(\Lambda)
\,, \ \Lambda\subseteq T_{\sqrt{2}\varepsilon(L)}(\tilde{\Gamma}(\omega,\lambda,\theta))\,,
\]
with $v_1\to V$ as $L\to \infty$. 
\end{thm}

Here $T_{\sqrt{2}\varepsilon}$ stands for the tube of radius $\sqrt{2}\varepsilon$ around a curve. 
That is it contains the points of the plane that are at a distance smaller than or equal to 
$\sqrt{2}\varepsilon$ of the given curve.

\begin{rem}We remark, that if $v_0\in(0,1)$ is a given volume fraction of liquid, 
then one could tune the theorem a bit to 
obtain solutions for this prescribed volume with probability close to one. This can be done by taking 
$V=v_0/(1-\nu\pi R)$ and $R$ sufficiently small, since the expected number of grains under $\Lambda$ is 
$\nu L^2 V$ and $v_0$ is related to $v_1$ by (\ref{v0:vs:v1}).
\end{rem}

Next, we define partitions of $[0,L]\times[( V - \varepsilon_{0})L,(V + \varepsilon_{0})L]$ into 
small squares of equal size. 
In each of such squares that intersect $\Lambda$ we formulate a percolation-like problem. 

\begin{defin}\label{Q:def}
Let 
\begin{equation}\label{Restricted:domain}
{\cal Q}:=[0,L]\times[(V - \varepsilon_{0})L, (V + \varepsilon_{0})L]\,.
\end{equation}
Let $\varepsilon>0$, depending on $L$, such that $B\ll \varepsilon \ll 1$ and such that $1/\varepsilon$, 
$2\varepsilon_0/\varepsilon\in\mathbb{N}$. 

We define $\{Q\}_{Q\in I_\varepsilon}$ to be a partition of ${\cal Q}$ 
into squares of size $\varepsilon L \times \varepsilon L$ such that if $Q$, $\tilde{Q}\in I_\varepsilon$ then 
$Q\cap \tilde{Q} =\emptyset$ if $Q\neq \tilde{Q}$, unless $Q$ and $\tilde{Q}$ have one edge or one vertex in common.

We let $M:=\varepsilon(L) \,L$ for simplicity.
\end{defin}

Clearly, $\# I_{\varepsilon}= 2\varepsilon_0/\varepsilon^2<\infty$. We also observe that this definition requires 
$\varepsilon_0\in\mathbb{Q}$.

We now consider a subdivision of each $Q$ in smaller sub-squares of size of order one. 
We assume that the size is exactly one, for that we either need to assume that $L\in\mathbb{N}$ or to
rescale the domain. We henceforth adopt the first choice without loss of generality.
  
\begin{defin}\label{SR:def} Assume that $L\in\mathbb{N}$
Let $\varepsilon(L)$ and $Q\in I_\varepsilon$ as in Definition~\ref{Q:def}, 
then we let $Q=\cup_{\kappa\in J} S_\kappa$ where $\{S_\kappa\}_{\kappa\in J}$ is a partition of $Q$ 
into closed squares of size one that cover $Q$ and such that their interior sets are disjoint. 
 
Moreover, for a $s\in(0,1)$ and for each $\kappa\in J$, 
we let $R_\kappa \subset S_\kappa$ be a closed squared with $|R_\kappa| =s$ centered in $S_\kappa$. 
\end{defin}

Observe that $\# J$ increases as $L\to \infty$, but $|S_\kappa|=1$, and this value
equals the average distance between centers (and is of the same order of the minimal radius of curvature 
in this domain). 
Then we can guarantee that if there are grains in neighbouring $S_\kappa$'s, 
then there is a connection by a solution $\gamma$ of (\ref{S1E1}). 
In this way the proof reduces to estimating the probability of finding grains in the $S_\kappa$'s 
that can form a connection between the given portions of $\partial Q$ that intersect $\Lambda$. 
This problem is reminiscent of a site percolation problem in two dimensions 
(see e.g. \cite{stauffer}, \cite{grimmett}), 
although the probabilities assigned to the sites, as we explain below, are weakly correlated.

In fact, we shall restrict the connections to grains that fall in adjacent $R_\kappa$'s. 
This is because we have to guarantee that elemental components do not intersect each other, 
this is possible by reducing $|R_\kappa|=s$, as we shall see. But first, we have the following 
result that guarantees the existence of $\gamma$'s joining grains in adjacent sites.
This result is independent of $\varepsilon$.

\begin{lem}\label{Q:construct}
\begin{enumerate}
\item Let $S_\kappa$ and $S_{\kappa'}$ such that $S_\kappa\cap S_{\kappa'}\neq \emptyset$ 
with one edge in common. Then if there exist $\xi\in R_\kappa $ and $\tilde{\xi}\in R_{\kappa'}$ then, 
there exists a $L_0>0$ such that for all $L\geq L_0$ there exists a  
solution of (\ref{S1E1}) that joins $\partial_{R} B(\xi)$ and $\partial B_{R}(\tilde{\xi})$.

\item There exists $s_0\in(0,1)$ and $\alpha_{max}(R)$ with $\alpha_{max}(R)\to \pi/2$ as $R\to 0$, 
such that if $\kappa$, $\kappa'$ and $\tilde{\kappa}\in\{1,\dots,M\}^2$ have a vertex in common 
the arc that joins a grain with center in $\xi\in R_\kappa$ to a grain with center in $\xi'\in R_{\kappa'}$ 
do not intersect the arc thar joins the grain $B_R(\xi')$ to a grain with center $\xi\in R_{\tilde{\kappa}}$. 
\end{enumerate}
\end{lem}
\begin{proof}
{\it Proof of (i):} Observe that the centers $\xi$ and $\tilde{\xi}$ satisfy that 
$\|\xi-\tilde{\xi}\|\leq \sqrt{s+2s^{\frac{1}{2}}+1}<2$. On the other hand, 
$|H|<a\varepsilon_0$, in particular the radius of curvature of a solution joining these grains, say $r$, 
has $|r|_\infty>1/a\varepsilon_0>2$.  

If there is a solution $\gamma$ that joins $\partial B(\xi,R)$ and $\partial B(\tilde{\xi},R)$, 
then the variation of the curvature of $\gamma$, $\Delta H < a|\xi_2-\tilde{\xi}_2|/L<2a/L$ 
(where $\Delta H=|H_{max}(\gamma)-H_{min}(\gamma)|$ is the difference of the maximum curvature 
and the minimum curvature of $\gamma$). This means that such a $\gamma$ is very close to a 
circle with large radius of curvature.

On the other hand, there are straight lines that join $\partial B(\xi,R)$ and $\partial B(\tilde{\xi},R)$ 
with any prescribed contact angle. 
Indeed, there are two parallel lines that are tangential to $\partial B(\xi,R)$ and $\partial B(\tilde{\xi},R)$. 
One has contact angle $\pi/2$ and the other has contact angle $-\pi/2$. Then for any given contact angle, 
there is a segment parallel to these ones that joins $\partial B(\xi,R)$ and $\partial B(\tilde{\xi},R)$ 
with that contact angle. This is also true for circumferences of large enough radius, being the smallest 
such circumference the one that has the two grains inscribed at the largest possible distance; 
i.e. one with diameter $\|\xi-\tilde{\xi}\|+2R$.

We recall that $H(s)=\beta'(s)$ where $\beta\in[-\pi,\pi]$ is the angle from $\mathbf{e}_{1}$ 
to the tangent vector $\mathbf{t}$, thus $|\beta'(s)|<a\varepsilon <1/2$ and $|\beta''(s)|<a/L$. 
This implies that for $L$ sufficiently large we can find a solution $\gamma$ close to a large circumference that joins 
$\partial B(\xi,R)$ and $\partial B(\tilde{\xi},R)$ with a prescribed contact angle by adjusting the value of $C$ 
and the height $x_2$ (see (\ref{A1E1})).

{\it Proof of (ii):} Let us see that reducing if necessary the size of each $R_\kappa$ the solutions have no 
self-intersections.
We concentrate in three adjacent sites, say $(1,1)$, $(1,2)$ and $(2,2)$ for definiteness and assume that 
there is exactly one grain in each of $R_{(1,1)}$, $R_{(1,2)}$ and $R_{(2,2)}$, with centers $\xi_{(1,1)}$, 
$\xi_{(1,2)}$ and $\xi_{(2,2)}$, respectively. 
Other situations can be recovered by translation and by a $\pi/2$, $\pi$ or $3\pi/2$ rotation. 

Let $b_s= (1-\sqrt{s})/2$ and let $\phi_s$ be the angle that the segments that join the grain in 
$R_{(1,2)}$ to the other two grains form. The length of such segments is $\sqrt{s+1}\in(1,\sqrt{2})$.
Self-intersections may occur in several ways, we aim to avoid then in the situations where the angle between 
joining elemental components if narrowest. 
The minimum possible $\phi_s$, if $s$ is fixed, is achieved 
when the centers of the grains are $(1-b_s,1-b_s)$ for the one in $R_{(1,1)}$, $(b_s,2-b_s)$ for the one in 
$R_{(1,2)}$ and $(1+b_s,1+b_s)$ for the one in $R_{(2,2)}$ (see Figure~\ref{no:self:int}). This gives:
\begin{equation}\label{theta:s}
\cos \phi_s=\frac{2(1-2b_s)}{(1-2b_2)^2+1}=\frac{2\sqrt{s}}{s+1}\,.
\end{equation}
This function is an increasing concave function for $s\in(0,1)$, this implies that 
$\phi_s\in (0,\pi/2)$ is a decreasing function with $\phi_s=\pi/2$ if $s=0$ and $\phi_s=0$ if $s=1$.

First we assume that $R=0$, in that case we do not have to worry about contact angles. 
Suppose now that the grains are joined by elemental components that are circular arcs 
tangent at $(b_s,2-b_s)$, i.e. at the grain in $R_{(1,2)}$. This gives the minimum radius of curvature 
\begin{equation}\label{max:radius}
r_{min}(s) = \frac{\sqrt{s+1}}{2\sin(\phi_s/2)}\,.
\end{equation}
In view of (\ref{theta:s}), we obtain that $r_{min}(s)$ is an increasing function with $r_{min}(0)=1/\sqrt{2}$ and 
$\lim_{s\to 1}r_{min}(s)=\infty$.
This implies that if the radius of the circular arcs is smaller than $r_{min}$ they intersect,  
i.e. this implies that for each $s\in (0,1)$, if 
\begin{equation}\label{H:restriction}
|H|< \min\{2\sin(\phi_s/2)/\sqrt{s+1},\varepsilon_0 a\}
\end{equation} 
the arcs do not intersect. In particular, if we take $0<s<s_0$ where $s_0$ is such that 
$r_{min}(s_0)=(a\varepsilon_0)^{-1}$, this ensures that $|H|<\varepsilon_0 a$ and non-self 
intersecting arcs.

Suppose now that $R>0$, with $R<1/2$. We can assume that $H$ is nearly constant 
for fixed $0<s<1$, so that (\ref{H:restriction}) is satisfied. Let $r$ denote 
the radius of the corresponding circumference and $(c_1,c_2)$ the coordinates of its center. 

The situation is symmetric with respect to the diagonal with negative slope of $S_{(1,2)}$. 
The arc that joins $B(\xi_{(1,2)})$ to $B(\xi_{(2,2)})$ (following orientation) has it center 
on the line $x_2=\frac{1}{\sqrt{s}}(x_1-1) + 2$. Let us assume, for definiteness, that $H>0$ 
at the contact point; in that case, $c_1>b_s+1/2$. We recall that if $\rho_c<-\pi/4$, then, 
by symmetry, the arc that joins $B(\xi_{(1,2)})$ to $B(\xi_{(2,2)})$
intersects the arc that joins $B(\xi_{(1,2)})$ to $B(\xi_{(1,1)})$. 
Thus we have to find a sufficient condition on $s$ that guarantees that $\rho_c>-\pi/4$, 
this will also imply a restriction on $\alpha$.

We can derive the radius of curvature $r_{min}$ for which the contact point has 
$\rho_c=-\pi/4$ and $\alpha=0$. This gives
\[
r_{min}(s)= \frac{\sqrt{(1+s)^2 s+(3-s)^2}-4\sqrt{s}R}{4\sqrt{s}\tan(\phi_s/2)}
\]
which is an increasing function of $s$. 
The value $\alpha=0$ gives the maximum of the contact angle for this radius of curvature. 
Then, if $s\in(0,s_0)$ with $s_0$ such that $r_{min}(s_0)=(a\varepsilon_0)^{-1}$, we obtain that 
$r>r_{min}(s)$, since $|H|<a\varepsilon_0$. Moreover, for such $r$ given, there is a restriction 
on $\alpha$, namely $\alpha\leq \alpha_{max}(s,r,R)<0$ where $\alpha_{max}(s,r,R)$ is 
the contact angle formed by the arc that intersects $B(\xi_{(1,2)})$ 
at the point with $\rho_c=-\pi/4$.

\begin{figure}[hhh]
\begin{center}
\includegraphics[width=0.75\textwidth,height=.35\textwidth]{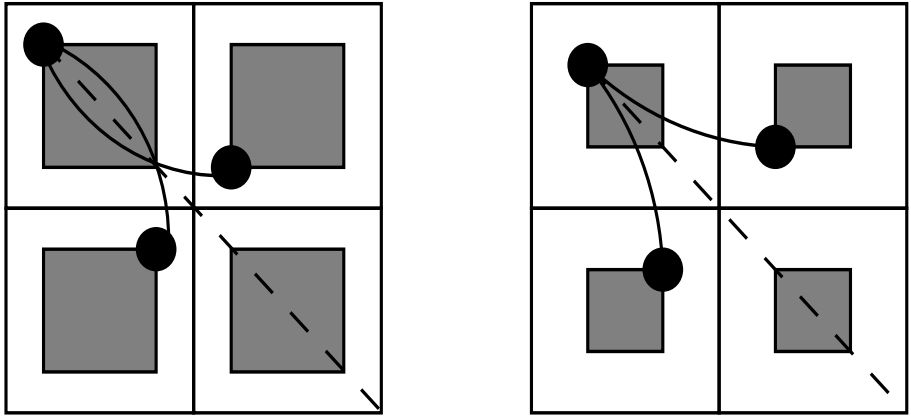}
\end{center}
\caption{Schematic picture showing how decreasing $s$ may avoid self intersecting arcs. 
In this case the centers of the grains are placed at the corners of $R_{\kappa}$ which makes the narrowest angle between arcs.}
\label{no:self:int}
\end{figure}

\end{proof}

In order to formulate the problem we identify each $Q$ with a set of 
coordinates in $\mathbb{N}^2$, namely: 

\begin{defin}\label{identify}
Let $\varepsilon$ and $M=\varepsilon L$ as in Definition~\ref{Q:def}. 
Adopting percolation terminology, we refer to each sub-square $S_\kappa$ and, without loss of generality, to
$R_\kappa$ of the partition defined in Definition~\ref{SR:def} as a {\bf site}. 

We identify every $Q\in I_\varepsilon$, via a suitable homeomorphism, with the square 
$[0,M]^{2}\subset\mathbb{R}^{2}$. 
We also denote by $\partial Q$ the union of sites that intersect the boundary of $Q$.

Finally, the sites $S_\kappa$ of $Q=[0,M]^{2}$ are identified with the coordinate of their top-right corner; 
if we identify $S_\kappa$ with $\kappa=(i,j)$, $i,j\in\{  1,2,3,\dots,M\}$, this denotes 
a site with top-right corner in $(i,j)$. Without loss of generality, $\kappa=(i,j)$ will be identified with $R_\kappa$.

We also identify the sides of $\partial Q$ as follows:
\begin{eqnarray*}
  \partial_{1} = \{1,2,3,\dots,M\} \times \{ 1 \} \,,
\ \partial_{2} =  \{ 1 \}  \times \{1,2,3,\dots,M\} \,,\\
 \partial_{3} =   \{1,2,3,\dots,M\} \times \{ M \} \,,
\ \partial_{4} = \{ M \}  \times \{1,2,3,\dots,M\}\,.
\end{eqnarray*}

\end{defin}

\begin{figure}[hhh]
\centering
\mbox{
\subfigure[The domain ${\cal Q}$ divided into squares $Q$]
{
\includegraphics[width=0.40\textwidth,height=.40\textwidth]{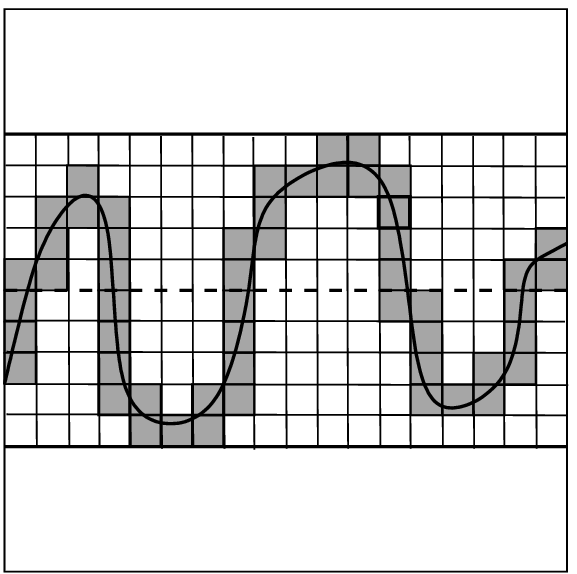}
\label{skecht1}
}
}
\hspace{1.40cm}
\mbox{
\subfigure[A square $Q$ divided into sub-squares $S_\kappa$]
{
\includegraphics[width=0.40\textwidth,height=.40\textwidth]{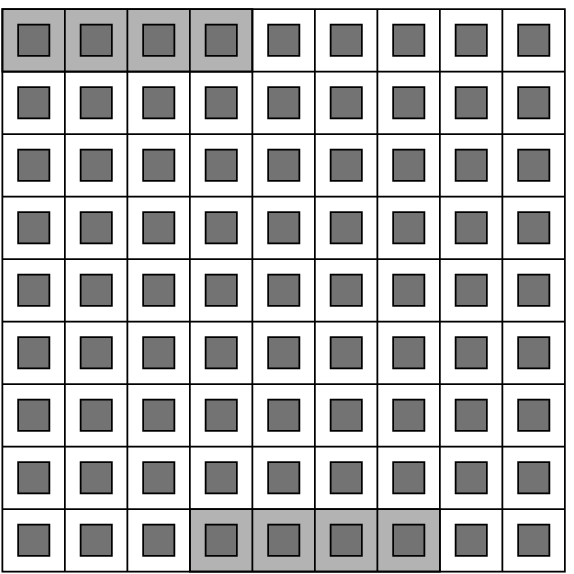}
\label{skecht2}
}
}
\caption{This figure shows a schematic picture of the division of the domain 
${\cal Q}$ that we use in the proof of Theorem~\ref{percolation:theo}. 
In \subref{skecht1} we show a curve $\Lambda$ contained in ${\cal Q}$ 
(thick solid line). The dashed line represents $x_2=\lambda$ and ${\cal Q}$ 
appears divided into squares $Q$ of size $(\varepsilon L)^2$ with 
$B\ll \varepsilon\ll 1$. The shaded $Q$'s correspond the ones that intersect $\Lambda$. 
In \subref{skecht2} we show an arbitrary $Q$ subdivided into squares $S_\kappa$ 
of size one. Each $S_\kappa$ contains the shaded region $R_\kappa$ with 
$|R_\kappa|=s\in (0,1)$. Here a portion of $\partial Q$ that intersects $\Gamma$ is highlighted.
}
\label{skecht}
\end{figure}

In the following we define what is meant by $Q$ is 'connected' or 
'not connected'. This depends on whether the sites contain grains or not, and whether 
these are connected throughout $Q$.

\begin{defin}\label{connectedness} 
Let $\omega\in\Omega_{\nu}(L)$, where we identify $\omega=\{\xi_j\}_{j=1}^{\nu L^2}$, then:
\begin{enumerate}
\item We say that a site $\kappa\in \{1,2,3,\dots,M\}^{2}$ is {\it open} 
if there exists at least one $j\in\{1,\dots,\nu L^2\}$ such that 
$\xi_{j}\in R_{\kappa}$. 

\item Similarly, we say that a site $\kappa\in \{1,2,3,\dots,M\}^{2}$ is 
{\it closed} 
if $\omega\cap R_{\kappa}=\emptyset$. 

\item The following random variable determines whether a site is open or closed:
\begin{equation}
\tau(\kappa;\omega)=\left\{
\begin{array}[c]{l}%
1\,,\quad\mbox{if}\quad\sum_{j=1}^{\nu L^{2}}\chi_{R_{\kappa}}(\xi_{j})>0\\  \\
0\,,\quad\mbox{if}\quad\sum_{j=1}^{\nu L^{2}}\chi_{R_{\kappa}}(\xi_{j})=0\,,
\end{array}
\right.  \label{random:vars}%
\end{equation}
here $\chi_{A}$ denotes the characteristic function for some $A\subset [0,M]^{2}$. 
\end{enumerate}

\end{defin}

The following definition sets when open and closed sites are connected, 
see e.g. \cite{russo2} for similar definitions 
in two-dimensional site percolation.

\begin{defin}
\label{Conn} 
Let $\omega\in\Omega_{\nu}(L)$, we say that two sites $\kappa_0$, $\kappa\in \{1,2,3,\dots,M\}^{2}$ are
\begin{enumerate}
\item  \textbf{adjacent}, and we write $\kappa_0\mathcal{R}(\omega)\kappa$, 
if $\operatorname*{dist}(\kappa_0,\kappa)=1$ and $\tau(\kappa_0,\omega)= \tau(\kappa,\omega)=1$. 

\item \textbf{connected} if there exists a finite set of distinct sites, $\kappa
_{1}$,..., $\kappa_{j}\in \{1,2,3,\dots,M\}^{2}$, such that every two such consecutive
sites are adjacent, and $\kappa_0\mathcal{R}(\omega) \kappa_{1}$ and
$\kappa_{j}\mathcal{R}(\omega)\kappa$. A set $\{\kappa_0,\kappa_{1}, \dots, \kappa_{j},\kappa\}$ 
with this property is called a \textbf{chain}.

\item \textbf{$*$-adjacent}, and we write $\kappa_0\mathcal{R}^*(\omega)\kappa$, 
if $\operatorname*{dist}(\kappa_0,\kappa) \leq \sqrt{2}$ and $\tau(\kappa_0,\omega)= \tau(\kappa,\omega)=0$.

\item \textbf{$*$-connected} if there exists a finite set of distinct sites, $\kappa
_{1}$, ..., $\kappa_{J}\in \{1,2,3,\dots,M\}^{2}$, such that every two such consecutive
sites are $*$-adjacent, and $\kappa_0\mathcal{R}^*(\omega) \kappa_{1}$ and
$\kappa_{j}\mathcal{R}^*(\omega)\kappa$. A set $\{\kappa_0,\kappa_{1}, \dots, \kappa_{j},\kappa\}$ 
with this property is called a \textbf{$*$- chain}.
\end{enumerate}

\end{defin}

The next definition gives the concept of regions being {\it connected} and $*$-{\it connected}.

\begin{defin}
Let $A_1\subset \partial_k$ and $A_2\subset\partial_l$ for $k\neq l$ and $k,l=1,2,3,4$ 
be topologically connected regions of $\partial Q$. Then, for a given $\omega\in\Omega_\nu(L)$, we say that
\begin{enumerate}
\item  $A_1$ is \textbf{connected} to $A_2$ if there exists a chain that intersects them.

\item $A_1$ is \textbf{$*$-connected} to $A_2$ if there exists a $*$-chain that intersects them.
\end{enumerate}

\end{defin}

We give a probability space to the configurations of sites induced by $\mu_{\nu}$:

\begin{defin}
We endow $\Omega_M:=\{0,1\}^Q$ where $Q\equiv\{1,2,3,\dots,M\}^{2}$ with the structure of a probability space. 
We define the probability $\tilde{\mu}_M$ as follows, for any $A\in\Omega_{M}$ we set:
\begin{equation}
\tilde{\mu}_M(A) = \mu_\nu\big(  \left\{  \omega\in\Omega_\nu(L):\ \tau
(Q;\omega) \in A\right\}  \big) \,.\label{Prob2}%
\end{equation}
\end{defin}

For example, let ${\cal C}_\kappa=\{\kappa\in\{1,\dots,M \}^2 \ \mbox{is closed}\}$ then
\begin{equation}\label{prob:c}
\tilde{\mu}_M({\cal C}_\kappa)= \mu_{\nu}\big(\{\omega\in \Omega_\nu(L): \ \tau(\kappa;\omega)=0\}\big)= 
\left(  1-\frac{|R_{\kappa}|}{L^{2}}\right)^{\nu L^{2}}
\end{equation}
and let ${\cal O}_\kappa=\{\kappa\in\{1,\dots,M \}^2 \ \mbox{is open}\}$, then 
\[
\tilde{\mu}_M({\cal O}_\kappa)= \mu_{\nu}\big(\{\omega\in \Omega_\nu(L): 
\ \tau(\kappa;\omega)=1\}\big)= 1- \tilde{\mu}_M({\cal C}_\kappa)\,.
\]

We observe, that in contrast to the classical site percolation scenario, the probabilities of two different 
sites of being open are not mutually independent, because for finite $L$ the number of grains is finite, and that 
influences the presence of finding more or less grains in different sites.

On the other hand, since $|R_{\kappa}|=s$ for all $\kappa\in\{1,\dots,M\}^2$ then 
$\tilde{\mu}_M({\cal C}_\kappa)=q(L)\in[0,1]$ and $\tilde{\mu}_M({\cal C}_\kappa)=1-q(L)$ 
are independent of $\kappa\in\{1,\dots,M\}^2$. Also, from (\ref{prob:c}) we get 
\begin{equation}\label{def:q}
q(L) \longrightarrow  e^{-\nu s}\quad\mbox{as} \ L\to\infty\,.
\end{equation}
Thus, in the limit $L\to\infty$, we find a classical site percolation
scenario; in this limit the number of grains in the domain 
(as well as the domain size) goes to infinity, making independent of each 
other the events $\{\omega\in\Omega_{nu}(L): \ \tau(\kappa;\omega)=1\}$. This means that there is a threshold value of $\nu$, 
say $\nu_0$ sufficiently large, 
such that the probability of finding an infinite cluster of closed sides is zero, i.e. that the dual event, 
finding an infinite cluster of open sites, 
has probability one, see e.g. \cite{russo2}, \cite{ LT}, \cite{Toth}.

In our setting, we have the following percolation result:
\begin{thm}[Percolation result]\label{percol}
Let $A_1\subset \partial_k$ and $A_2\subset\partial_l$ for $k\neq l$ and 
$k,l=1,2,3,4$ be disjoint and topologically connected regions of $Q$ of size 
$N_1$ and $N_2$, respectively, with $N_1$, $N_2\in (M/4, M/2)$.
Let $C(A_1,A_2)\subset \Omega_M$ be the set of configurations of sites for which $A_1$ 
is connected to $A_2$. Then, there exist $\nu_0>0$ and $L_0>0$ such that for all 
$\nu\geq \nu_0$ and $L\geq L_0$ for $\varepsilon(L)$ as in Definition~\ref{identify} 
such that
\[
\tilde{\mu}_M\big(C(A_2,A_2)\big)\geq 1-e^{-K\,\varepsilon(L)\,, L}
\]
for some $K>0$ independent of $L$. 
\end{thm}

In order to prove this result we show that 
the probability of the dual or complementary events tends to zero. 
We identify the complementary events 
with the aid of the following topological lemma:
\begin{lem}\label{connectedVSnonconnected}
Let $A_1\subset \partial_k$ and $A_2\subset\partial_l$ for $k\neq l$ and 
$k,l=1,2,3,4$ be disjoint and topologically connected regions of $\partial Q$. 
Let $A_1^c$ and $A_2^c$ be the topologically connected components of 
$\partial Q \backslash (A_1\cup A_2)$, then $A_1$ and $A_2$ are not connected 
if and only if $A_1^c$ and $A_2^c$ are $*$-connected.
\end{lem}

Results analogous to Lemma~\ref{connectedVSnonconnected} are classical in 
percolation theory and are related to the Jordan curve theorem (in this case 
the exterior of $A_1\cup A_2$ in $Q$ is $*$-connected 
if and only if $A_1$ and $A_2$ are not connected in $Q$). 
Notice that the lemma is independent of any probability choice. 
The proof can be adapted from the literature, 
see e.g. Bollobas~\cite{Bollobas}.

We can now show Proposition~\ref{percol}.

\begin{proof}[Proof of Proposition~\ref{percol}]
In view of Lemma~\ref{connectedVSnonconnected} and taking into account that 
$A_1$ and $A_2$ have length in $(M/4,M/2)$, the crucial percolation result that 
we have to prove is that the probability of a $*$-chain to reach 
a distance of order $M$ tends to zero.

We first observe that given $\kappa\in \{1,2,\dots,M\}^2$, let 
$\mbox{Chain}(\kappa,j)\subset \Omega_M$ 
with $j\in\mathbb{N}$ and $j\leq M^2$, be the set of site configurations for 
which there exist a $*$-chain of the form $\{\kappa,\kappa_1,\dots,\kappa_{j-1}\}$, then
\begin{equation}\label{j-chain}
\tilde{\mu}_M\big( \mbox{Chain}(\kappa,j) \big)\leq 9^{j}\left(1-\frac{j s}{L^2}\right)^{\nu L^2}\,.
\end{equation}
This is because the probability that a given set of $j$ sites are closed is 
\[
\left(1-\frac{J s}{L^2}\right)^{\nu L^2}\,,
\]
and there are at most 
$9^j$ $*$-chains with that length (since there are at most $9$ sites $*$-adjacent to any given closed site).

Let $J_h\in \{1,\dots,M\}$ the maximum horizontal distance from $A_1$ to $A_2$ and let 
$ J_v\in \{1,\dots,M\}$ the maximal vertical distance from $A_1$ to $A_2$. 
Then $J^*= \max\{J_h,J_v\}$ satisfies $J^*\in (M/4,M]$.

Let $C^*(A_1^c,A_2^c)\subset \Omega_M$ be the set of site configurations such that 
$A_1^c$ and $A_2^c$ are $*$-connected, then for $L$ sufficiently large
\begin{equation}\label{est:muB}
\tilde{\mu}_M\big( C^*(A_1^c,A_2^c)\big) \leq M \sum_{J^{*}\leq j \leq J^{*} M} 9^j 
\left( 1  -  \frac{j s}{L^2} \right)^{\nu L^2} < M\sum_{J^{*}\leq j \leq J^{*} M} 9^j e^{-\nu j s}\,.
\end{equation} 
Then, taking $\nu$ large enough we obtain that $\log(9)-\nu s< -1/2$, thus from 
(\ref{est:muB}) we get
\begin{equation}\label{est:muB2}
\tilde{\mu}_M\big( C^*(A_1^c,A_2^c)\big) \leq M J^{*}(M-1) e^{-\frac{J^{*}}{2}}\,.
\end{equation} 
We write $J^{*}= l M$ for certain value $l\in(1/4,1]$, then
\[
M^2 J^{*} e^{-\frac{J^{*}}{2}} = l M^3 e^{-\frac{l}{2}M}\,.
\]
This and the definition of $M$ (Definition~\ref{SR:def}) concludes the proof.

\end{proof}

We can now show Theorem~\ref{percolation:theo}.

\begin{proof}[Proof of Theorem~\ref{percolation:theo}]
Lemma~\ref{Q:construct} implies that we can determine the probability of a 
$\Gamma$ intersecting $Q$ through portions of $\partial Q$ that intersect 
$\Lambda$ by computing the probability of finding grain centers in the sub-boxes 
$R_\kappa$. This Lemma also give the restriction in the contact angle.  

We let $\nu_0$ and $L_0$ and $\varepsilon(L)$ be as in Theorem~\ref{percol} and 
 assume that $\nu\geq \nu_0$ and $L\geq L_0$. 
Let $J_L(\Lambda)\subset \mathbb{N}$ such that if 
$i\in J_L(\Lambda)$ then $Q_i\subset I_{\varepsilon(L)}$ satisfies that 
$Q_i\cup \Lambda\neq \emptyset$. For every $i \in J_L(\Lambda)$, let also $S_1(Q_i)$, 
$S_2(Q_i)\subset \partial Q_i$ such that $S_j(Q_i)\cap \Lambda \neq \emptyset$ 
for $j=1,2$ and with length in $(M/4,M/2)$ as in Theorem~\ref{percol}. 

Then, by Theorem~\ref{percol}, 
\[
\tilde{\mu}_M(C(S_1(Q_i),S_2(Q_i)))\geq 1-e^{-K\,\varepsilon(L)\, L}\,,
\] 
i.e. for all $\mathcal{U}_i\subset\Omega_{\nu}(L)$ compatible with $C(S_1(Q_i),S_2(Q_i))$, 
$\mu_\nu\big( \mathcal{U}_i\big)  \geq 1-1/(\varepsilon(L)\, L)^K$.
On the other hand
\[
\mu_\nu\big(\cap_{i\in J_L(\Lambda)} \mathcal{U}_i\big)\geq 
1- \sum_{i\in J_L(\Lambda) } \mu_\nu\left( \mathcal{U}_i\right)\geq 
1-\#  J_L(\Lambda) e^{-K \varepsilon(L)\, L}
\]
Observe that $1/\varepsilon(L)< \# J_L(\Lambda)<1/\varepsilon^2(L)$, 
thus taking $\varepsilon(L)=L^{-p}$ with $0<p <1$ 
we obtain the result with 
\[
\delta_\varepsilon(L)=1-L^{2p} e^{-K  L^{1-p}}\,.
\]

\end{proof}

\paragraph{Acknowledgements}
The authors acknowledge support of the Hausdorff Center for Mathematics of the University of Bonn 
as well as the project CRC 1060 The Mathematics of Emergent Effects, that is funded through the 
German Science Foundation (DFG). This work was also partially supported by the Spanish projects DGES MTM2011-24-10 
and MTM2011-22612, the ICMAT Severo Ochoa project SEV-2011-0087 and the Basque Government project IT641-13.

\section*{Appendix~}\label{appendix}

\appendix

\section{Proof of Lemma~\ref{geometric:lemma}}\label{sec:proofgeolemma}

We divide the proof in several steps. First we recall that we are assuming 
\begin{equation}\label{regime:hyp}
L \gg \min\{ \sqrt{B},R^{-1}\} \quad  \mbox{with}\quad B\to \infty \quad \mbox{and}\quad R\to 0\,.
\end{equation}

A rough estimate on the total length of a solution $\Gamma$ that joins $m$ 
centers is achieved by thinking of an interface having the longest possible length. We show that this worst-case 
scenario is given by an interface $\Gamma$ having as many as possible folds elongating nearly horizontally 
from one vertical boundary to the other. It is clear that the longest elemental components of such interface will 
occur near $x_2=v_1 L$ and that their union connects very few grains, since for such pieces $C$ is close to $1$ 
(see Lemma~\ref{C:estimates}).

Taking into account that the volume is fixed and that the grains are uniformly distributed 
(and do not see gravity) there must be folds on both sides of $x_{2}=v_{1} L$. We can distinguish two types of interfaces, 
those that zigzag starting above $x_{2}=v_{1} L$ and those starting below it.

Let us make the structure of such hypothetical interface solutions more precise. 

As we did in Definition~\ref{def:lr:rl} we also distinguish different types of elemental components 
and parts of the interface. First we distinguish those parts above $x_2=v_1 L$, that we indicate 
with a subindex $+$ (they have positive curvature), those below $x_2=v_1 L$, that we indicate with 
a superindex $-$ since (negative curvature), and those that cross $x_2=v_1 L$, 
that have curvature changing sign across $x_2=v_1 L$, and that we indicate with the superindex $c$. 
Then, we can define, skipping the dependency on $(\omega,\lambda,\theta)$ for simplicity,
\begin{equation}\label{worstcase:pieces}
\Gamma= \Gamma^{+} \cup \Gamma^{-} \cup \Gamma^{c} 
\quad \mbox{with} \quad 
\Gamma^{+} = \bigcup_{i\in J^{+}} \gamma^{+,i}\,,
\quad 
\Gamma^{-} = \bigcup_{j\in J^{-}} \gamma^{-,j}\,,
\quad\mbox{and}\quad
\Gamma^{c} = \bigcup_{k\in J^{c}} \gamma^{c,k}\,,
\end{equation}
and the set of indexes $J^{+}$, $J^{-}$ and $J^{c}$ are defined in the obvious way, they are disjoint and 
$J^{+}\cup J^{-} \cup J^{c}=J$.

We distinguish, also in Definition~\ref{def:lr:rl}, two types of elemental components, those that, following its 
orientation, go from left to right (with subscript $lr$), and those that go from 
right to left (subscript $rl$).

We further decompose $\Gamma$ into horizontal levels: 
\begin{defin}[Horizontal level]\label{def:level} 
Given an interface $\Gamma$, we say that a portion 
$\tilde{\Gamma}\subset\Gamma$ is a horizontal level of $\Gamma$ if 
\begin{enumerate}
\item
\[
\min_{\xi_{j}\in\tilde{\Gamma}} \mbox{dist}(\xi_{j},x_1=0)< 2R+1\quad \mbox{and}\quad 
\min_{\xi_{j}\in\tilde{\Gamma}} \mbox{dist}(\xi_{j},x_1=L)< 2R+1\,,
\]
\item And, if we let $d_{min}(\tilde{\Gamma})=\min_{x\in \tilde{\Gamma}}\mbox{dist}(x,x_2 =v_1 L )$ and  
$d_{max}(\tilde{\Gamma})=\min_{x\in \tilde{\Gamma}}\mbox{dist}(x,x_2 =v_1 L )$, then 
\begin{equation}\label{vertical:restriction}
 d_{max}(\tilde{\Gamma}) - d_{min}(\tilde{\Gamma})\leq 2R + \frac{4}{\sqrt{B}}\,\frac{1}{1+ \sqrt{B} d_{min}(\tilde{\Gamma})} 
\end{equation}
\end{enumerate}
\end{defin}

We now give a definition of a curve that zigzags horizontally:
\begin{defin}[Zig-zag interface]\label{def:zigzag}
A solution $\Gamma$ of (\ref{S1E1}) with Young conditions is a zigzag interface if:
\begin{enumerate}
\item There exists a set of consecutive levels $\Gamma_{0}$ that contains $\Gamma^{c}$. One 
of these levels, that we denote by $\Gamma_{0}^{0}$, must 'cross' $x_{2}=v_{1} L$, in the sense that 
if $\xi_{c}$ is the first and $\zeta_{c}$ the last contact point connected by $\Gamma_{0}^{0}$, then 
$\xi_{c}$ and $\zeta_{c}$ lie on opposite sides of $x_{2} = v_{1} L$.

\item There exists $N^+$ and $N^- \in \mathbb{N}$ with $N^-$, $N^+\geq 0$ such that
\[
\Gamma = \left(\bigcup_{i=1\dots N^{+}}\Gamma^{+}_{i}\right)\cup\left( \bigcup_{i=1\dots N^{-}} 
\Gamma^{-}_{i}  \right)\cup \Gamma_{0}
\]
where $\Gamma^{+}_{i}\in \Gamma^{+}$ with $i\in\{1,\dots,N^{+}\}$ and 
$\Gamma^{-}_{j}\in \Gamma^{-}$ with $i\in\{1,\dots,N^{-}\}$ are horizontal levels of $\Gamma$, 
indexed in decreasing order with proximity to $x_2=v_1 L$. 
\end{enumerate}
\end{defin}

\begin{rem}\label{Clessthan1}
We distinguish four type of interfaces, namely:
\begin{enumerate}
\item Type I: the level $\Gamma_{0}^{0}$ has elemental components of the form $\gamma_{lr}$ and 
crosses from positive to negative curvature.

\item Type II: the level $\Gamma_{0}^{0}$ has elemental components of the form $\gamma_{rl}$ and 
crosses from positive to negative curvature.

\item Type III: the level $\Gamma_{0}^{0}$ has elemental components of the form $\gamma_{lr}$ and 
crosses from negative to positive curvature.

\item Type IV: the level $\Gamma_{0}^{0}$ has elemental components of the form $\gamma_{rl}$ and 
crosses from negative to positive curvature.

\end{enumerate}

This definition does not exclude the possibility that other levels rather than the ones contained in 
$\Gamma_{0}$ have elemental components of the form $\gamma_{lr}$ with $C<1$. 

We shall write 
\[
\Gamma_{0} = 
\left(\bigcup_{j\in J_{0}^{+}} \gamma^{+,j} \right)\cup 
\left(\bigcup_{j\in J_{0}^{-}} \gamma^{-,j} \right)\cup
\left(\bigcup_{j\in J_{0}^{c}} \gamma^{c,j} \right)
\]
where, depending on the type of interface, $J_{0}^{+}$ and $J_{0}^{-}$ might be empty. 

Observe that there exists at most three consecutive levels in $\Gamma_{0}$: $\Gamma_{0}^{0}$, one above it 
$\Gamma_{0}^{1}$, and one below it, $\Gamma_{0}^{-1}$. The levels $\Gamma_{0}^{1}$ and $\Gamma_{0}^{-1}$ do 
not cross $x_{2}=v_{1}L$ in the sense that $\Gamma_{0}^{0}$ does, but individual elemental components of them 
can. This is clear from the orientation of the elemental components and Definition~\ref{def:zigzag} 
(specify in at least one case of the types above). 

We might write:
\[
\Gamma_{0}^{k} = 
\left(\bigcup_{j\in J_{0}^{k,+}} \gamma^{+,j} \right)\cup 
\left(\bigcup_{j\in J_{0}^{k,-}} \gamma^{-,j} \right)\cup
\left(\bigcup_{j\in J_{0}^{k,c}} \gamma^{c,j} \right) \quad \mbox{for} \quad k=-1,0,1\,. 
\]
\end{rem}

We shall first derive Lemma~\ref{geometric:lemma} for zigzag interfaces in which $\Gamma_{0}$ has three 
levels. Then we prove that given the set of interfaces connecting $m$ grains, then these interfaces maximise 
the length.

Observe that such $\Gamma_{0}$'s connect at least $4$ centers, 
a pathological case given by a $\Gamma_{0}$ composed 
of just three elemental components, each being a level of $\Gamma_{0}$: $\Gamma_{0}^{0}=\{\gamma_{rl}^{c}\}$, 
oscillating around $x_{2}=v_{1}L$, $\Gamma_{0}^{1}=\{\gamma_{lr}^{+}\}$, just above $x_{2}=v_{1}L$, 
and $\Gamma_{0}^{-1}=\{\gamma_{lr}^{-}\}$.

Let us first derive some elementary properties of zigzag interfaces:
\begin{lem}\label{basic:zigzag:IandIII}
If an interface $\Gamma$ is a zigzag interface of type I or III and such that its $\Gamma_{0}$ has exactly 
three levels, then:
\begin{enumerate}
\item $N^{+}$ and $N^{-}$ are odd numbers. 

\item 
\[
\Gamma^{\pm}_{2i-1} = \bigcup_{j\in J_{2i-1}^{\pm}} \gamma_{lr}^{\pm,j} \quad \mbox{and}\quad 
\Gamma^{\pm}_{2i} = \bigcup_{j\in J_{2i}^{\pm}} \gamma_{rl}^{\pm,j}\quad \mbox{with}\quad 
i=1,\dots,\frac{N^{\pm}+1}{2}\,.
\] 

\item The index sets $\{J_{i}^{+}\}$ are disjoint with $\cup_{i=1\dots N^{+}} J_{i}^{+}\subset J^{+}$
and also $\{J_{i}^{-}\}$ are disjoint with $\cup_{i=1\dots N^{-}}J_{i}^{-}\subset J^{-}$. 
These sets might be empty.

\item If $\Gamma$ connects $m$ grains then $\gamma_{1}\in \Gamma_{N^{+}}^{+}$ and 
$\gamma_{m+1}\in \Gamma_{N^{-}}^{-}$ if $\Gamma$ is of type $I$ and $\gamma_{1}\in \Gamma_{N^{-}}^{+}$ and $\gamma_{m+1}\in \Gamma_{N^{+}}^{+}$ if $\Gamma$ is of type $III$.
\end{enumerate}
\end{lem}
We leave the analogous lemma for type II and IV to the reader. Figure~\ref{zigzags} shows a skecht of the four types of zigzag interfaces with three levels in $\Gamma_0$.
\begin{figure}[hhh]
\begin{center}
\includegraphics[width=0.60\textwidth,height=.60\textwidth]{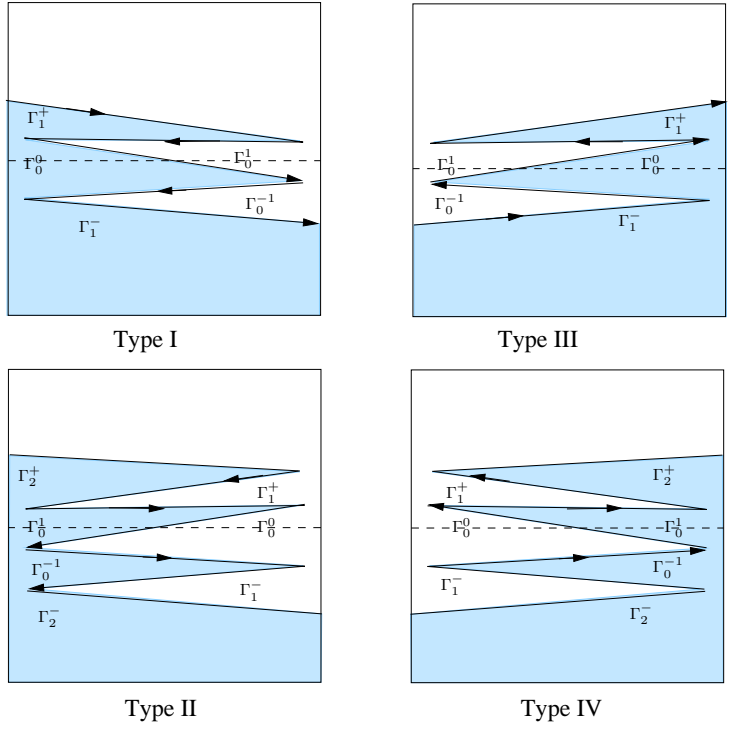}
\end{center}
\caption{This picture shows a sketch of the four types of zig-zag solutions of Definition~\ref{def:zigzag}. The dashed line represents $x_2=v_1 L$, 
the solid lines represent the interface. In the examples shown, the central horizontal level, $\Gamma_0$, has three parts ($\Gamma_0^0$, that crosses $x_2=v_1 L$, 
$\Gamma_0^1$, above $x_2=v_1 L$, and $\Gamma_0^{-1}$, below $x_2=v_1 L$) each represented by a straight line. In the examples of type I and III there is just one horizontal
level above and below $\Gamma_0$, and in the examples of  type II and IV there are two above and two below.}   
\label{zigzags}%
\end{figure}

\begin{proof}[Proof of Lemma~\ref{geometric:lemma}]
Assume there exists an $\omega\in \Omega_\nu(L)$ and a $\Gamma(\omega,\lambda,\theta)$ that is a zig-zag 
interface and connects $m$ grains where $m\geq  3$ and has three levels in $\Gamma_0$. it is clear any other 
interface that connects $m$ grains is necessarily of the same length or shorter, the length of the interface 
in a strip of the size of $\Gamma_0$ being at most of the order of $3L$, outside the strip the length between 
connections of grains is of the same order as for a zigzag interface, according to 
Definition~\ref{def:level} {\it (ii)} and Corollary~\ref{corol1}.

We can now prove the result for zigzag interfaces that join grains near $x_2=v_1L $ with the maximal possible length.
Let us assume that $\Gamma$ is of type $I$. Then, it has three levels in $\Gamma_0$ with three elemental components of the 
form $\gamma^c$ whose length is of the order of $L$. Let
\[
n_{0}= \# J_{0}\,, \quad n_{i}^{+} = \# J_{i}^{+} \quad \mbox{and} \quad n_{i}^{-} = \# J_{i}^{-}\,,
\]
where $J_0=J_0^+\cap J_0^-\cap J_0^c$, and $J_{i}^{\pm}$ are as in Lemma~\ref{basic:zigzag:IandIII}. 
Let us make abuse of notation by saying that $\# J^c_0=3$, so that we index in this group only 
the pieces $\gamma^c$ with length of order $L$, and the other pieces $\gamma^c$ of smaller order of length
we index them in $J_0^+$ is they are above the of the former, and in $J_0^-$ if they are below. 
We then let $n_{0}^\pm= \# J_{0}^\pm$.

We recall that lemmas~\ref{C:estimates} and \ref{C:estimates:2} imply that the other elemental components 
$\gamma^c$ or $\gamma^{lr}$ and $\gamma^{rl}$ in $\Gamma_0$ have horizontal length comparable to the horizontal 
of elemental components in the even horizontal layers.

Thus we write
\begin{equation}\label{total:estimate}
\sum_{i=0}^{m}{\cal D}_{i} = \sum_{ \substack{i=0\dots m\\ i\in J_{0}^c}  }{\cal D}_{i} + \sum_{ \substack{i=0\dots m\\ i\in J_{0}^+}  }{\cal D}_{i} 
+ \sum_{ \substack{i=0\dots m\\ i\in J_{0}^-}  }{\cal D}_{i}     + \sum_{ \substack{i=0\dots m\\ i\notin J_{0}}} {\cal D}_{i} \,.
\end{equation}

Using Definition\ref{def:level} (\ref{vertical:restriction}), we can estimate the first term in (\ref{total:estimate}) by
\begin{equation}
\sum_{\substack{ i=0\dots m\\ i \in J_{0}^c}} {\cal D}_{i}\leq 
 3\left( L^{2} +\left( 2R + \frac{4}{\sqrt{B}}\right)^{2} \right)^{\frac{1}{2}}\leq 
3\left( L+ 2R +\frac{4}{\sqrt{B}}\right)\,,
\label{j0:estimate}
\end{equation}
because $d_{min}(\Gamma_0^0)=d_{min}(\Gamma_0^{\pm1})=0$.

We write ${\cal D}_{i}=(({\cal D}_{i}^{h})^{2}+({\cal D}_{i}^{v})^{2})^{\frac{1}{2}}\leq {\cal D}_{i}^{h}+{\cal D}_{i}^{v}$ where 
${\cal D}_{i}^{h}=|\xi_{i}^{(1)}-\xi_{i}^{(1)}|$ and ${\cal D}_{i}^{v}=|\xi_{i}^{(2)}-\xi_{i}^{(2)}|$. In particular,
\begin{equation}\label{vertical:estimate}
\sum_{ \substack{i=0\dots m\\ i\in J_{0}^+}  }{\cal D}_{i}^v 
+ \sum_{ \substack{i=0\dots m\\ i\in J_{0}^-}  }{\cal D}_{i}^v  
+ \sum_{\substack{ i=0\dots m\\  i\notin J_{0}}} D_{i}^{v}\leq (N^{+}+N^{-}+2)\left(2R + \frac{4}{\sqrt{B}}\right)\,.
\end{equation}
Where we still need to control $N^++N^-$ in terms of $m$.

We let $n_{min}^{\pm,even} = \min_{j=0,\dots,(N^{\pm}-1)/2}\{ n_{2j}^{\pm}\}$ and 
$n_{min}^{\pm,odd} = \min_{j=1,\dots,(N^{\pm}+1)/2}\{ n_{2j-1}^{\pm}\}$, and define 
$n_{min}^{even}=\min\{ n_{min}^{even,+},n_{min}^{even,-}\}$ and 
$n_{min}^{odd}=\min\{ n_{min}^{odd,+},n_{min}^{odd,-}\}$. 
It is clear that $n_{min}^{odd}\geq 2$ and that $n_0\geq 4$, then
\[
m+1 = \sum_{i=1}^{N^+}n_i^⁺ +\sum_{i=1}^{N^-} n_i^- + n_0= \sum_{i=0}^{N^+}n_i^⁺ +\sum_{i=0}^{N^-} n_i^- + 3
\geq (N^+ + 1) +(N^- + 1) +\sum_{j=0}^{(N^+-1)/2}n_{2j}^⁺ + \sum_{j=0}^{(N^--1)/2}n_{2j}^- + 3 
\]
and therefore,
\begin{equation}\label{m:estimate} 
m+1\geq N^+ + N^- + 5 +  n_{min}^{even}  \frac{N^+ + N^- + 2}{2}   \,.
\end{equation}
Which in particular gives that $m\geq N^++N^-+2$.

Let us get an estimate on the horizontal distances. 
Observe that 
\begin{equation}\label{par}
\sum_{\substack{ j=0,\dots,(N^{\pm}-1)/2 \\i\in J_{2j}^{\pm} }} {\cal D}_{i}^{h} 
\leq   n_{min}^{even} \frac{N^{\pm}+1}{2} \max_{\substack{j=1,\dots,(N^{\pm}-1)/2\\i\in J_{2j}}}{\cal D}_{i}^{h} \,. 
\end{equation}
On the other hand, since the absolute value of the curvature of the pieces increases with 
the distance to the horizontal, for each $j\in \{ 1,\dots,(N^{\pm}-1)/2\}$  
we have $\sum_{i\in J_{2j-1}^{\pm}} {\cal D}_{i}^{h} \leq 2 \sum_{i\in J_{2j}^{\pm}} {\cal D}_{i}^{h}$. Then, 
\begin{equation}\label{par:impar}
\sum_{ \substack{ j=0,\dots,\frac{N^{\pm}-1}{2} \\ i\in J_{2j}^{\pm}\cup J_{2j+1}^{\pm} }} {\cal D}_{i}^{h}\leq 3 
\sum_{\substack{j=0,\dots,\frac{N^{\pm}-1}{2}\\i\in J_{2j}} } {\cal D}_{i}^{h}\,.
\end{equation}
We can now estimate the horizontal levels using (\ref{par:impar}) and (\ref{par}), this gives
\begin{equation}\label{horizontal:estimate}
\sum_{ \substack{i=0\dots m\\ i\in J_{0}^+}  }{\cal D}_{i}^h + \sum_{ \substack{i=0\dots m\\ i\in J_{0}^-}  }{\cal D}_{i}^h 
+\sum_{\substack{ i=0\dots m\\  i\notin J_{0}}} {\cal D}_{i}^{h}  \leq 3 n_{min}^{even} \frac{N^{+}+N^{-}+2}{2} 
\max_{\substack{j=0,\dots,N^{\pm}-1)/2\\ i\in J_{2j}^{+}\cup J_{2j}^{-}}}{\cal D}_{i}^{h}  +2L\,,
\end{equation}
where the horizontal size of $\Gamma_{N^+}$ and of $\Gamma_{N^-}$ is estimated by $L$.

Using the estimates (\ref{j0:estimate}), (\ref{vertical:estimate}), (\ref{m:estimate}) and 
(\ref{horizontal:estimate}) in (\ref{total:estimate}) we get
\begin{equation}\label{pre:estimate}
\sum_{i=1}^{m}{\cal D}_{i}\leq  (m-2)\left(2R + \frac{4}{\sqrt{B}}\right)
+3 (m-4) \max_{i\in J_{2j}^{+}\cup J_{2j}^{-}}D_{i}^{h} + 5L\,.
\end{equation}

It is clear that the maximal distance between connected centers in the levels $\Gamma_{2i}^{\pm}$ 
is achieved at an elemental component $\gamma$ with $C=C^{\ast}$ in an elemental component of 
the form $\gamma_{rl}^{\pm}$. Although we can only guarantee that $C^\ast >-1$, we observe that
 the horizontal distances between centers are uniformly bounded. 
Namely, if $C^{\ast}>1$, Lemma~\ref{single:distance} {\it(ii)} implies  
\[
 \max_{i\in J_{2j}^{+}\cup J_{2j}^{-}}D_{i}
\leq 2R + \left(\frac{\pi^2+2}{B(C^\ast +1)}\right)^\frac{1}{2}\leq  2R + 
\left(\frac{\pi^2+2}{2B}\right)^\frac{1}{2}\,.
\]
If $0<C^{\ast}<1$ then Lemma~\ref{single:distance2} ({\it (ii)}) implies
\[
\max_{i\in J_{2j}^{+}\cup J_{2j}^{-}}D_{i}
\leq 2R + \left(\frac{\pi^2+2}{B(C^{\ast}+1)}\right)^\frac{1}{2}\leq
2R + \left(\frac{\pi^2+2}{B}\right)^\frac{1}{2}\,,
\]
and if $-1<C^{\ast}\leq 0$, then Lemma~\ref{single:distance3} implies
\[
\max_{i\in J_{2j}^{+}\cup J_{2j}^{-}}D_{i}
\leq 2R + \left(\frac{2(\pi^2+4)}{B}\right)^\frac{1}{2}\,.
\]

Thus, taking the maximum of these upper bounds we obtain that
\[
\sum_{i=1}^{m}D_{i}\leq  
(4m-14)\left(2R + \frac{K}{\sqrt{B}}\right)+ 5L
\]
where $K=\sqrt{2(\pi^2+4)}\approx 5.3$ and this gives (\ref{geometric}).

The proof for the other types of zigzag interfaces is similar and is left to the reader.
\end{proof}


\begin{thebibliography}{10}

\bibitem{batchelor}
G.~K. Batchelor.
\newblock {\em An Introduction to Fluid Dynamics}.
\newblock Cambridge University Press, 1967.

\bibitem{bear}
J.~Bear.
\newblock {\em Dynamics of Fluids in Porous Media}.
\newblock Science. American Elsevier Publishing Company, 1972.

\bibitem{Bollobas}
B.~Bollob{\'a}s and O.~Riordan.
\newblock {\em Percolation}.
\newblock Cambridge University Press, New York, 2006.

\bibitem{DirrDondl}
N.~Dirr, P.~W. Dondl, and M.~Scheutzow.
\newblock Pinning of interfaces in random media.
\newblock {\em Interfaces and Free Boundaries. Mathematical Modelling, Analysis
  and Computation}, 13(3):411--421, 2011.

\bibitem{FellerI}
W.~Feller.
\newblock {\em An introduction to probability theory and its applications.
  {V}ol. {I}}.
\newblock Third edition. John Wiley \& Sons, Inc., New York-London-Sydney,
  1968.

\bibitem{Finn}
R.~Finn.
\newblock {\em Equilibrium capillary surfaces}, volume 284 of {\em Grundlehren
  der Mathematischen Wissenschaften [Fundamental Principles of Mathematical
  Sciences]}.
\newblock Springer-Verlag, New York, 1986.

\bibitem{grimmett}
G.~Grimmett.
\newblock {\em Percolation}, volume 321 of {\em Grundlehren der Mathematischen
  Wissenschaften [Fundamental Principles of Mathematical Sciences]}.
\newblock Springer-Verlag, Berlin, 2 edition, 1999.

\bibitem{haines}
W.~B. Haines.
\newblock Studies in the physical properties of soil: V. the hysteresis effect
  in capillary properties, and the modes of moisture associated therewith.
\newblock {\em Journal of Agricultural Sciences}, 20, 1930.

\bibitem{JBS07}
W.~J{\"a}ger and B.~Schweizer.
\newblock Microscopic interfaces in porous media.
\newblock In {\em Reactive flows, diffusion and transport}, pages 555--565.
  Springer, Berlin, 2007.

\bibitem{LT}
T.~Luczak and J.~C. Wierman.
\newblock Critical probability bounds for two-dimensional site percolation
  models.
\newblock {\em J. Phys. A: Math. Gen}, 21:3131--3138, 1988.

\bibitem{Maloy92}
K.~J. M{\aa}l{\o}y, L.~Furuberg, and J.~Feder.
\newblock Dynamics of slow drainage in porous media.
\newblock {\em Phys. Rev. Lett.}, 1992.

\bibitem{Maloy96}
K.~J. M{\aa}l{\o}y, L.~Furuberg, and J.~Feder.
\newblock Intermittent behavior in slow drainage.
\newblock {\em Phys. Rev. E}, 1996.

\bibitem{Myshkis}
A.~D. Myshkis, V.~G. Babski{\u\i}, N.~D. Kopachevski{\u\i}, L.~A. Slobozhanin,
  and A.~D. Tyuptsov.
\newblock {\em Low-gravity fluid mechanics}.
\newblock Springer-Verlag, Berlin, 1987.

\bibitem{russo2}
L.~Russo.
\newblock On the critical percolation probabilities.
\newblock {\em Zeitschrift für Wahrscheinlichkeitstheorie und Verwandte
  Gebiete}, 56:229--237, 1981.

\bibitem{SP}
E.~Sanchez-Palencia.
\newblock {\em Non-Homogeneous Media and Vibration Theory}, volume 127 of {\em
  Lecture Notes in Physics 127}.
\newblock Springer, Berlin, 1980.

\bibitem{BS05_I}
B.~Schweizer.
\newblock Laws for the capillary pressure in a deterministic model for fronts
  in porous media.
\newblock {\em SIAM Journal on Mathematical Analysis}, 36:1489--1521
  (electronic), 2005.

\bibitem{BS05_II}
B.~Schweizer.
\newblock A stochastic model for fronts in porous media.
\newblock {\em Annali di Matematica Pura ed Applicata. Series IV},
  184:375--393, 2005.

\bibitem{BS07_III}
B.~Schweizer.
\newblock Averaging of flows with capillary hysteresis in stochastic porous
  media.
\newblock {\em European J. Appl. Math.}, 18(3):389--415, 2007.

\bibitem{stauffer}
D.~Stauffer and A.~Aharony.
\newblock {\em Introduction To Percolation Theory}.
\newblock CRC PressINC, 1994.

\bibitem{Tartar}
L.~Tartar.
\newblock {\em Convergence of the homogenization process}, chapter Appendix of
  \cite{SP}.

\bibitem{Toth}
B.~T\'oth.
\newblock A lower bound for the critical probability of the square lattice site
  percolation.
\newblock {\em Z. Wahscheinlichkeitstheorie verw. Gebiete}, 69:19--22, 1985.

\end{thebibliography}
\def\cprime{$'$}

\end{document}